\documentclass[11pt]{article}

\usepackage[margin=1in]{geometry}
\usepackage[T1]{fontenc}
\usepackage{array}
\usepackage{amsmath,amssymb,amsthm,mathtools}
\usepackage{tikz}
\usetikzlibrary{arrows.meta,calc,decorations.pathmorphing}
\usepackage[colorlinks=true,linkcolor=blue,citecolor=blue,urlcolor=blue]{hyperref}

\newtheorem{theorem}{Theorem}[section]
\newtheorem{definition}[theorem]{Definition}
\newtheorem{lemma}[theorem]{Lemma}
\newtheorem{proposition}[theorem]{Proposition}
\newtheorem{corollary}[theorem]{Corollary}
\newtheorem{remark}[theorem]{Remark}
\newtheorem{example}[theorem]{Example}

\newcommand{\C}{\mathbb C}
\newcommand{\CH}{\mathbb C\mathrm H}

\newcommand{\PH}{\mathbb P}
\newcommand{\Span}{\operatorname{span}}
\newcommand{\argg}{\operatorname{arg}}
\newcommand{\dd}{\,\mathrm d}

\title{Universal Hermitian Projective Calculus for \(\CH^2\)\\
\large Pair, Phase, Contact, and Kernel Invariants in Complex Hyperbolic Two-Space}
\author{Karl J. Kreder III}
\date{May 2026}

\begin{document}
\maketitle

\begin{abstract}
This paper develops an algebraic invariant calculus for complex hyperbolic
two-space \(\CH^2\).  The model is Hermitian projective rather than real
projective: \(\CH^2\) is the domain of negative complex lines in a Hermitian
vector space \(V\cong\C^{2,1}\), and its boundary is the projectivized null
cone.  Consequently the analogue of Wildberger's quadrance/spread language
cannot be a single real pair of invariants.  It splits into Hermitian pair
moduli, determinant quadrance, polar spread, triple-product phase,
cross-ratio laws, CR/contact boundary formulas, normalized-ball kernel
identities, and finite Hermitian incidence.  The development is intrinsic to
Hermitian projective geometry and uses only the standard complex-hyperbolic
strata: negative, null, and positive projective lines.
\end{abstract}

\section*{Preface}

Complex hyperbolic two-space is a classical object with several equally useful
faces.  It is a rank-one Hermitian symmetric space, a projective domain of
negative lines, a ball with its Bergman metric, and a geometry whose boundary
carries a CR/contact structure rather than an ordinary conformal sphere.  The
purpose of this manuscript is to place these familiar features into a single
algebraic calculus.  The starting point is only a three-dimensional vector
space with a Hermitian form of signature \((2,1)\).  From that datum one obtains
the negative, null, and positive projective strata, and from their pairings one
obtains the invariants developed below.

The guiding principle is that the Hermitian setting does not reduce to the
real hyperbolic pattern with complex coefficients inserted.  In the real
projective plane, a rational trigonometry can be organized around quadrance and
spread.  In \(\CH^2\), the same role is distributed among several coupled
quantities.  Pair moduli record the symmetric metric content of two lines.
Determinant quadrances and polar spreads record join, pole, and orthogonality
relations.  Triple products carry a phase, and that phase is the algebraic
source of the Cartan angular invariant.  On the null boundary, the same
Hermitian pairings degenerate into Heisenberg and contact expressions.  Thus
the calculus is not a longer notation for the classical metric theory; it is a
projective way of separating the metric, phase, polar, and boundary components
that coexist in complex hyperbolic geometry.

A second principle is denominator control.  Many standard expressions in
complex hyperbolic geometry are naturally written as ratios.  Ratios are
perfectly adequate on the open strata where their denominators are nonzero, but
they obscure how formulas behave under incidence, degeneration, and finite
field specialization.  For this reason the main identities are stated both as
projective quotients and, when useful, as cleared polynomial relations.  This is
what allows the same algebraic pattern to describe ordinary complex
hyperbolic points, positive poles of complex geodesics, boundary chains,
spinal spheres, kernel overlaps on the ball, and finite Hermitian unital
incidence.

The manuscript may be read in layers.  The opening sections fix the Hermitian
projective model and the basic pair invariants.  The middle sections develop
triangle trigonometry, phase, cross-ratios, complex geodesic poles, projective
unitary covariance, bisectors, spinal spheres, and the CR/contact atlas.  The
later sections connect the projective calculus with normalized ball kernels
and finite Hermitian geometries over fields with involution.  A reader mainly
interested in classical \(\CH^2\) may regard the algebraic invariants as
projective coordinates for familiar metric and CR objects.  A reader interested
in finite or symbolic models may instead treat the quotient formulas as
readouts of the underlying polynomial identities.

The result is intended as a textbook-style calculus rather than a catalog of
isolated formulas.  Each invariant is introduced with its projective scaling
law, its geometric meaning, and the identities that make it usable in
computations.  Classical distance, angle, chain, and kernel formulas reappear
as analytic interpretations of the algebraic data, not as the definitions from
which the algebra depends.

\newpage
\tableofcontents
\newpage

\section{Introduction}

Wildberger's universal hyperbolic geometry develops the real projective
hyperbolic plane from a symmetric bilinear form \cite{wildberger-uhg}.  Its
basic measurements, quadrance and spread, are rational projective invariants;
null points and null lines are not discarded, but become part of the algebra.
We use that as a methodological guide, not as a template to copy.

For \(\CH^2\), the governing form is Hermitian, not symmetric bilinear.
The replacement is therefore:
\[
\begin{array}{c}
  \text{real bilinear quadrance/spread} \\
  \leadsto \\
  \text{Hermitian pair, triple, cross-ratio, contact, and kernel invariants}.
\end{array}
\]

\begin{table}[ht]
\centering
\renewcommand{\arraystretch}{1.2}
\begin{tabular}{>{\raggedright\arraybackslash}p{0.34\textwidth}
                >{\raggedright\arraybackslash}p{0.50\textwidth}}
\hline
Real projective hyperbolic calculus & Hermitian \((2,1)\) calculus for \(\CH^2\)\\
\hline
Symmetric bilinear form & Hermitian form \(h(z,w)\), linear in the first slot and
conjugate-linear in the second\\
Projective points and null points & Negative, null, and positive complex lines:
\(\PH(V_-)\), \(\PH(V_0)\), \(\PH(V_+)\)\\
Quadrance & Pair modulus \(\eta_H\) and determinant quadrance
\(q_H=\Delta_H/(h(z,z)h(w,w))\)\\
Spread between lines & Polar spread, obtained by applying the same pair
determinant calculus to positive poles\\
Triangle cross laws & Three-point Hermitian Gram identity, cyclic spread law,
and Pythagoras specialization\\
No real analogue in a single quadrance/spread pair & Triple-product phase
\(T_H=-h_{12}h_{23}h_{31}\), the algebraic source of the Cartan angular
invariant\\
Boundary conic & Null boundary with CR/contact and Heisenberg structure\\
Finite rational models & Finite Hermitian incidence over fields with involution\\
\hline
\end{tabular}
\caption{A translation dictionary from the real bilinear universal
hyperbolic language to the Hermitian projective language used here.}
\label{tab:uhg-dictionary}
\end{table}

The standard background is already well established.  Parker's notes, for
example, treat Hermitian forms of signature \((2,1)\), the projectivized
negative-line model of \(\CH^2\), the ball and Siegel models, the Heisenberg
boundary compactification, CR boundary geometry, subspaces, and distance and
cross-ratio formulae \cite{parker-nchg}.  This paper takes that infrastructure
as background.  The additional aim here is the algebraic packaging: projective
pair/quadrance identities, triple-phase and cross-ratio quotient laws,
vertex-spread cross laws and Pythagoras identities, CR/contact readouts,
normalized-kernel phase/contrast factorization, and finite Hermitian incidence
identities, all organized as an intrinsic calculus.

The relation with the standard literature is therefore:
\begin{itemize}
\item The Hermitian \((2,1)\) projective model and null boundary are standard;
here they support projective pair invariants, determinant quadrance, and
null-join witnesses.
\item Positive polar vectors, complex lines, and boundary chains are standard;
here they are organized by polar spread, defined as the same pair calculus
applied to dual positive-pole data, and by a vertex-spread Hermitian cross
law for point triangles.
\item Koranyi--Reimann cross-ratios and metric formulas are standard; here the
focus is projective descent together with swap, inverse, double-swap, and
conjugation laws.
\item The Heisenberg compactification and CR boundary geometry are standard;
here they are packaged as boundary pair kernels, triple readouts,
contact-density laws, similarity covariance, chart transitions, and inversion
laws.
\item Ball and Siegel coordinate models over \(\C\) are standard; here they
feed normalized-ball kernel ratios and phase/contrast factorization.
\item Classical \(\CH^2\) is complex-analytic; here the calculus also points
toward finite Hermitian incidence over fields with involution, including the
order-two \(\mathbb F_4/\mathbb F_2\) model and the general rank-three
finite Hermitian unital counts.
\end{itemize}

\section{Hermitian projective model}

This section fixes the standard complex-hyperbolic conventions, following the
usual Hermitian projective model.  The algebraic contribution begins when
these standard strata are used as carriers for pair, phase, contact, kernel,
and finite-incidence invariants.

\begin{definition}[Hermitian projective model]
Let \(V\) be a complex vector space of dimension \(3\), equipped with a
Hermitian form \(h\) of signature \((2,1)\).  We use the convention
\[
  h(\lambda z,\mu w)=\lambda\overline{\mu}\,h(z,w).
\]
Define
\[
  V_-=\{z:h(z,z)<0\},\qquad
  V_0=\{z\ne0:h(z,z)=0\},\qquad
  V_+=\{z:h(z,z)>0\}.
\]
Then
\[
  \CH^2=\PH(V_-),\qquad
  \partial\CH^2=\PH(V_0).
\]
\end{definition}

\begin{definition}[Projective strata]
For a nonzero vector \(z\in V\), write \([z]\) for its complex projective
line.  We call \([z]\) negative, null, or positive according as
\[
  h(z,z)<0,\qquad h(z,z)=0,\qquad h(z,z)>0.
\]
Negative lines are points of \(\CH^2\), null lines are boundary points, and
positive polar vectors represent complex geodesics and boundary chains.
\end{definition}

\begin{definition}[Local positive frame]
Let \(x=[z]\in\CH^2\) be an interior point, so \(h(z,z)<0\).  The negative
line and its Hermitian orthogonal complement are
\[
  L_x=\C z,\qquad E_x=L_x^\perp.
\]
Since \(h\) has signature \((2,1)\), the restriction of \(h\) to \(E_x\)
is positive definite and
\[
  V=L_x\oplus E_x,\qquad E_x\cong\C^2.
\]
The compact local frame group at \(x\) is
\[
  U(E_x)\cong U(2),
  \qquad
  SU(E_x)\cong SU(2).
\]
\end{definition}

\begin{proposition}[Local frame representation dictionary]
The first natural complex representations attached to the local frame are
\[
\begin{array}{ccl}
E_x &\longleftrightarrow& \text{the fundamental \(2\)-dimensional
representation of \(SU(E_x)\)},\\[1mm]
\mathfrak{su}(E_x)_\C\cong\mathfrak{sl}(E_x)
&\longleftrightarrow& \text{the adjoint \(3\)-dimensional representation},\\[1mm]
\C &\longleftrightarrow& \text{the trivial \(1\)-dimensional representation}.
\end{array}
\]
Equivalently, in highest-weight notation for \(SU(2)\), these are the
irreducibles of highest weights \(1\), \(2\), and \(0\), respectively.
\end{proposition}

\begin{proof}
The positive frame \(E_x\) is a two-dimensional Hermitian vector space, so
the simple part of its unitary group is \(SU(E_x)\cong SU(2)\).  Its defining
action on \(E_x\) is the fundamental two-dimensional irreducible
representation.  The complexified adjoint action is the conjugation action
on
\[
  \mathfrak{su}(E_x)_\C=\mathfrak{sl}(E_x),
\]
which is the three-dimensional irreducible representation of \(SU(2)\).
Finally, scalar frame-invariant quantities transform through the trivial
one-dimensional representation.  These three representations are canonical
once the interior point \(x\) and hence the positive frame \(E_x\) have been
chosen.
\end{proof}

\begin{proposition}[Cartan tangent split at an interior point]
Let \(x=[z]\in\CH^2\) and let \(V=L_x\oplus E_x\) be the local splitting
above.  The Lie algebra of \(SU(2,1)\) admits the Cartan decomposition
\[
  \mathfrak{su}(2,1)=\mathfrak k_x\oplus\mathfrak p_x,
\]
where
\[
  \mathfrak k_x
  \cong
  \mathfrak{s}\bigl(\mathfrak u(L_x)\oplus\mathfrak u(E_x)\bigr)
  \cong
  \mathfrak u(2),
\]
and
\[
  \mathfrak p_x\cong \operatorname{Hom}_\C(L_x,E_x)
\]
as a real vector space.  Consequently
\[
  T_x\CH^2\cong \operatorname{Hom}_\C(L_x,E_x),
  \qquad
  \dim_\C T_x\CH^2=2,\quad \dim_{\mathbb R} T_x\CH^2=4.
\]
\end{proposition}

\begin{proof}
Choose an \(h\)-orthonormal basis adapted to \(V=L_x\oplus E_x\), so that
the Hermitian matrix of \(h\) is
\[
  J=\begin{pmatrix}-1&0\\0&I_2\end{pmatrix}.
\]
The Lie algebra \(\mathfrak{su}(2,1)\) consists of trace-zero complex
matrices \(X\) satisfying
\[
  X^\ast J+JX=0.
\]
Writing \(X\) in block form relative to \(L_x\oplus E_x\), this condition is
equivalent to
\[
  X=
  \begin{pmatrix}
    i\alpha & b^\ast\\
    b & A
  \end{pmatrix},
  \qquad
  \alpha\in\mathbb R,\quad A\in\mathfrak u(2),\quad
  i\alpha+\operatorname{tr}A=0,
\]
with \(b\in\C^2\).  The block-diagonal part \(b=0\) is the stabilizer of
the line \(L_x\); it is
\[
  \mathfrak{s}\bigl(\mathfrak u(L_x)\oplus\mathfrak u(E_x)\bigr)
  \cong\mathfrak u(2).
\]
The off-diagonal part is determined by \(b\in E_x\), equivalently by a
complex linear map \(L_x\to E_x\) after choosing a nonzero vector in
\(L_x\).  Thus the complement is naturally identified, as a real vector
space, with \(\operatorname{Hom}_\C(L_x,E_x)\).  This complement is the
tangent representation of the symmetric space at \(x\), giving the stated
complex and real dimensions.
\end{proof}

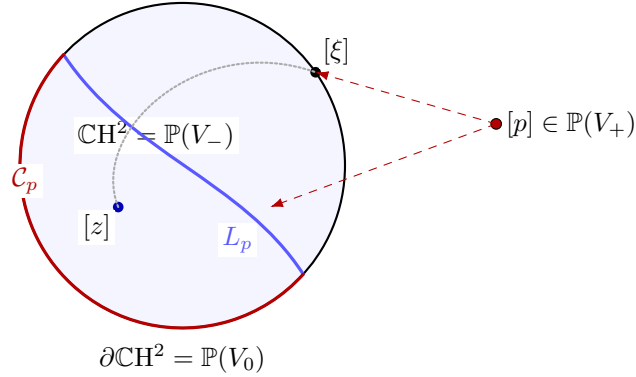
\begin{figure}[ht]
\centering
\begin{tikzpicture}[
    x=1cm,y=1cm,
    line cap=round,
    line join=round,
    >=Latex,
    label/.style={font=\small, fill=white, inner sep=1.5pt},
    note/.style={font=\small}
  ]
  \def\R{2.15}

  \fill[blue!4] (0,0) circle (\R);
  \draw[thick] (0,0) circle (\R);
  \node[label] at (-0.35,0.40) {\(\CH^2=\PH(V_-)\)};
  \node[label] at (0,-2.55) {\(\partial\CH^2=\PH(V_0)\)};

  \coordinate (Z) at (-0.85,-0.55);
  \fill[blue!70!black] (Z) circle (2pt);
  \node[label, below left=1pt] at (Z) {\([z]\)};

  \coordinate (Xi) at ({\R*cos(35)},{\R*sin(35)});
  \fill[black] (Xi) circle (2pt);
  \node[label, above right=1pt] at (Xi) {\([\xi]\)};

  \coordinate (A) at ({\R*cos(137)},{\R*sin(137)});
  \coordinate (B) at ({\R*cos(-42)},{\R*sin(-42)});
  \draw[very thick,blue!65]
    (A) .. controls (-0.70,0.15) and (0.78,-0.18) .. (B);
  \node[label, text=blue!65] at (0.73,-0.98) {\(L_p\)};

  \draw[very thick,red!70!black]
    (A) arc[start angle=137,end angle=318,radius=\R];
  \node[label, text=red!70!black] at (-2.08,-0.20) {\(\mathcal C_p\)};

  \coordinate (P) at (4.15,0.55);
  \draw[fill=red!70!black] (P) circle (2pt);
  \node[label, right=2pt] at (P) {\([p]\in\PH(V_+)\)};
  \draw[dashed,red!70!black,->] (P) -- (1.16,-0.55);
  \draw[dashed,red!70!black,->] (P) -- (Xi);

  \draw[densely dotted,thick,gray!65]
    (Z) .. controls (-1.25,0.45) and (0.22,1.78) .. (Xi);
  \node[note] at (0,2.70) {mixed data: interior point, boundary point, positive pole};
\end{tikzpicture}
\caption{Schematic mixed configuration.  A negative line \([z]\) is an
interior point, a null line \([\xi]\) is a boundary point, and a positive pole
\([p]\) determines both the complex geodesic \(L_p\) and its boundary chain
\(\mathcal C_p\).  The equations \(h(z,p)=0\) and \(h(\xi,p)=0\) are routed
through different strata.}
\label{fig:mixed-triangle}
\end{figure}

\section{Uniform algebraic layers and stratum domains}

The calculus uses the complex model as its main geometric picture, but the
algebraic expressions are deliberately organized so that finite Hermitian
models use the same formulas.  This requires separating four layers.

\begin{description}
\item[Starred-field projective algebra.]
Let \(K\) be a field with involution \(\lambda\mapsto\lambda^\ast\), and let
\(h\) be a Hermitian form over a rank-three \(K\)-space.  The scalar norm is
\[
  N(\lambda)=\lambda\lambda^\ast.
\]
Pair numerators, pair denominators, two-point Gram determinants, triple
products, and cross-ratio numerators and denominators are defined at this
level.  They are polynomial or rational Hermitian expressions with explicit
projective scale weights.

\item[Polarity-stratified algebra.]
Over a general field with involution there may be no order and therefore no
literal negative/positive inequality.  The replacement is a scale-stable null
predicate together with a chosen anisotropic or positive-pole predicate.  This
is the correct algebraic setting for boundary-chain incidence \(h(x,p)=0\)
and for finite Hermitian unital incidence.

\item[Analytic complex \(\CH^2\).]
When \(K=\C\) and \(h\) has signature \((2,1)\), the chosen predicates become
the familiar strata
\[
  V_-,\qquad V_0,\qquad V_+.
\]
This layer adds distance readouts, Cartan angular readouts, the Heisenberg
boundary chart, CR/contact forms, and analytic kernel formulas.

\item[Finite Hermitian specialization.]
For a quadratic finite-field extension \(K'/K\), the involution is the
Frobenius over \(K\), and the norm is \(x x^\ast\).  The same projective
Hermitian formulas give finite null points, positive-pole chains, norm-fiber
row charts, and the rank-two Hermitian unital incidence package.
\end{description}

\begin{proposition}[Layer descent]
Every invariant in the starred-field projective layer specializes to the
complex analytic and finite Hermitian layers after adding the appropriate
extra structure.  The complex layer adds ordered sign strata and CR/contact
analysis; the finite layer replaces ordered signs by scale-stable
null/anisotropic polarity strata and finite norm-fiber incidence.
\end{proposition}

\begin{proof}
The underlying expressions are built from \(h\), the involution, scalar
multiplication, and projective rescaling.  Their scale weights are algebraic:
non-null pair quantities divide by the matching diagonal weight, while triple
phase data quotient by norm scalars.  Adding a complex signature or a finite
Frobenius norm supplies extra interpretations of the same expressions, but it
does not change their projective descent laws.
\end{proof}

\begin{definition}[Field-uniform data]
Field-uniform CH2 data consist of the following four pieces.
\[
\begin{array}{ll}
\text{A. starred projective data} &
(K,\ast,V,h),\quad h(\lambda z,\mu w)=\lambda\mu^\ast h(z,w),\\
\text{B. projective strata} &
\text{scale-stable null, non-null, and pole predicates},\\
\text{C. analytic readouts} &
\text{ordered signs, CR/contact forms, kernels, and limits over }\C,\\
\text{D. finite readouts} &
\text{Frobenius norm fibers and finite Hermitian incidence.}
\end{array}
\]
Only A and B are part of the algebraic core.  C and D are specializations:
the analytic layer adds inequalities and differential/contact readouts, while
the finite layer replaces inequalities by polarity strata and norm-fiber
counts.
\end{definition}

\begin{definition}[Shared-determinant rank-pair group]
Let
\[
  G_{3,2}
  =
  U(3)\times_{\det,U(1),\det}U(2)
  =
  \{(A,B)\in U(3)\times U(2):\det A=\det B\}.
\]
This is the compact unitary group obtained by coupling a rank-three unitary
carrier and a rank-two unitary frame by a shared determinant.  It is an
auxiliary determinant-locking group attached to the rank pair \((3,2)\), not
the holomorphic isometry group of \(\CH^2\).
\end{definition}

\begin{theorem}[Shared-determinant quotient]
There is a natural isomorphism of compact Lie groups
\[
  G_{3,2}
  \cong
  \frac{SU(3)\times SU(2)\times U(1)}{\mu_6},
\]
where \(\mu_6=\{\lambda\in U(1):\lambda^6=1\}\) is embedded by
\[
  \lambda\longmapsto
  (\lambda^{-2}I_3,\lambda^{-3}I_2,\lambda).
\]
\end{theorem}

\begin{proof}
Define
\[
  \Phi:SU(3)\times SU(2)\times U(1)\longrightarrow U(3)\times U(2)
\]
by
\[
  \Phi(U,V,\lambda)=(\lambda^2U,\lambda^3V).
\]
For the first factor,
\[
  \det(\lambda^2U)=(\lambda^2)^3\det U=\lambda^6,
\]
and for the second factor,
\[
  \det(\lambda^3V)=(\lambda^3)^2\det V=\lambda^6.
\]
Thus \(\Phi\) lands in \(G_{3,2}\).

The map is onto.  Given \((A,B)\in G_{3,2}\), write
\[
  \delta=\det A=\det B\in U(1).
\]
Choose \(\lambda\in U(1)\) with \(\lambda^6=\delta\).  Then
\[
  U=\lambda^{-2}A,\qquad V=\lambda^{-3}B
\]
satisfy
\[
  \det U=\lambda^{-6}\det A=1,\qquad
  \det V=\lambda^{-6}\det B=1.
\]
Hence \(U\in SU(3)\), \(V\in SU(2)\), and
\[
  \Phi(U,V,\lambda)=(A,B).
\]

It remains to compute the kernel.  If
\[
  \Phi(U,V,\lambda)=(I_3,I_2),
\]
then
\[
  U=\lambda^{-2}I_3,\qquad V=\lambda^{-3}I_2.
\]
The determinant conditions \(U\in SU(3)\) and \(V\in SU(2)\) give
\[
  1=\det U=(\lambda^{-2})^3=\lambda^{-6},
  \qquad
  1=\det V=(\lambda^{-3})^2=\lambda^{-6}.
\]
Thus \(\lambda^6=1\).  Conversely, if \(\lambda^6=1\), then
\[
  (\lambda^{-2}I_3,\lambda^{-3}I_2,\lambda)
\]
lies in \(SU(3)\times SU(2)\times U(1)\) and maps to
\((I_3,I_2)\).  Therefore
\[
  \ker\Phi
  =
  \{(\lambda^{-2}I_3,\lambda^{-3}I_2,\lambda):\lambda^6=1\}
  \cong \mu_6.
\]
The first isomorphism theorem gives the displayed quotient.
\end{proof}

\begin{definition}[Central residue coordinates]
Let \(A_1\in\mathbb Z/6\mathbb Z\) denote the exponent of the shared
\(\mu_6\)-coordinate in the quotient presentation of \(G_{3,2}\).  Let
\(Y_2\in\mathbb Z/2\mathbb Z\) denote the exponent of the central
\(\mu_2\)-coordinate of the rank-two frame projection.  The common
sixth-root lattice embeds the frame coordinate by
\[
  Y_2\longmapsto 3Y_2\in\mathbb Z/6\mathbb Z.
\]
The associated central residue is
\[
  Q_{3,2}
  =
  \frac{A_1}{6}+\frac{Y_2}{2}
  =
  \frac{A_1+3Y_2}{6}
  \quad\in\quad \frac{1}{6}\mathbb Z/\mathbb Z.
\]
\end{definition}

\begin{proposition}[Rank-pair residue lattice]
Let \(t\in\mathbb Z/3\mathbb Z\) be the rank-three central exponent and
let \(w\in\mathbb Z/2\mathbb Z\) be the rank-two central exponent.  In the
common \(\mu_6\)-lattice define
\[
  Y_6\equiv -2t+3w\pmod 6.
\]
If \(k=\pm1\) denotes the two frame orientations, represented in the
same lattice by \(3k\), then the four first central residues are
\[
  Q_{3,2}(t,w,k)
  =
  \frac{Y_6+3k}{6}.
\]
For the rank-three singlet class \(t=0\) and the first rank-two class
\(w=1\), this gives
\[
  0,\qquad -1.
\]
For the rank-three fundamental class \(t=1\) and \(w=1\), this gives
\[
  \frac23,\qquad -\frac13.
\]
\end{proposition}

\begin{proof}
The center of \(SU(3)\) is \(\mu_3\), so its primitive exponent is
represented in the common \(\mu_6\)-lattice by multiplication by \(2\).
The center of \(SU(2)\) is \(\mu_2\), so its primitive exponent is
represented by multiplication by \(3\).  The sign convention in the
quotient embedding gives the combined central coordinate
\[
  Y_6\equiv -2t+3w\pmod 6.
\]
The two frame orientations differ by the two representatives \(k=\pm1\),
which contribute \(3k\) on the same sixth-root lattice.  Therefore
\[
  Q_{3,2}=\frac{Y_6+3k}{6}.
\]
For \(t=0,w=1\), one has \(Y_6\equiv3\pmod6\).  Using the representative
\(Y_6=-3\) gives
\[
  k=+1:\quad \frac{-3+3}{6}=0,\qquad
  k=-1:\quad \frac{-3-3}{6}=-1.
\]
For \(t=1,w=1\), one has \(Y_6\equiv -2+3=1\pmod6\), hence
\[
  k=+1:\quad \frac{1+3}{6}=\frac23,\qquad
  k=-1:\quad \frac{1-3}{6}=-\frac13.
\]
These are residues of the central quotient lattice; no metric or
analytic structure is used.
\end{proof}

\begin{proposition}[Projection quotients of the rank pair]
The shared-determinant group \(G_{3,2}\) has two natural projection
quotients:
\[
  U(3)\cong\frac{SU(3)\times U(1)}{\mu_3},
  \qquad
  U(2)\cong\frac{SU(2)\times U(1)}{\mu_2}.
\]
These two \(U(1)\)-coordinates are not independent inside \(G_{3,2}\);
they are projections of the single determinant coordinate
\[
  \det A=\det B.
\]
\end{proposition}

\begin{proof}
For rank \(n\), the standard map
\[
  SU(n)\times U(1)\longrightarrow U(n),\qquad
  (S,\lambda)\longmapsto \lambda S
\]
is surjective.  If \(\lambda S=I_n\), then
\[
  S=\lambda^{-1}I_n,\qquad 1=\det S=\lambda^{-n}.
\]
Thus the kernel is \(\mu_n\), embedded by
\[
  \lambda\longmapsto(\lambda^{-1}I_n,\lambda),
  \qquad \lambda^n=1,
\]
and the first isomorphism theorem gives
\[
  U(n)\cong (SU(n)\times U(1))/\mu_n.
\]
Taking \(n=3\) and \(n=2\) gives the two displayed projection quotients.

Inside \(G_{3,2}\), however, the projected unitary factors are constrained by
\[
  \det A=\det B.
\]
Therefore the scalar coordinate seen in the \(U(3)\)-projection and the
scalar coordinate seen in the \(U(2)\)-projection are two readings of the
same determinant coordinate.  The independent central periods are \(\mu_3\)
and \(\mu_2\), while their common comparison lattice is \(\mu_6\).
\end{proof}

\begin{lemma}[Scalar image for \(SU(3)\)-types]
Let \(W\) be a finite-dimensional complex representation of \(SU(3)\).  The
space of scalar images of \(W\) is
\[
  \operatorname{Hom}_{SU(3)}(W,\C),
\]
where \(\C\) carries the trivial representation.  This space is nonzero
exactly on the trivial summand of \(W\).  In particular, if \(W\) is an
irreducible nontrivial \(SU(3)\)-representation, then
\[
  \operatorname{Hom}_{SU(3)}(W,\C)=0.
\]
At the infinitesimal level, the determinant character contributes no first
order scalar on \(\mathfrak{su}(3)\):
\[
  \dd(\det)_I(X)=\operatorname{tr}X=0
  \qquad (X\in\mathfrak{su}(3)).
\]
\end{lemma}

\begin{proof}
Since \(SU(3)\) is compact, every finite-dimensional complex representation is
completely reducible.  Decompose
\[
  W\cong \C^{\oplus m}\oplus W',
\]
where \(W'\) has no trivial irreducible summand.  By Schur's lemma, an
equivariant map from a nontrivial irreducible summand to the trivial
representation is zero, while each trivial summand contributes one scalar
projection.  Therefore
\[
  \operatorname{Hom}_{SU(3)}(W,\C)\cong \C^m.
\]
This proves the stated criterion and the irreducible nontrivial case.

For the determinant statement, let \(X\in\mathfrak{su}(3)\).  The standard
differential identity for determinant gives
\[
  \left.\frac{\dd}{\dd t}\right|_{t=0}\det(I+tX)=\operatorname{tr}X.
\]
The Lie algebra \(\mathfrak{su}(3)\) consists of traceless skew-Hermitian
matrices, hence \(\operatorname{tr}X=0\).
\end{proof}

\begin{remark}[Separation from the complex-hyperbolic isometry group]
The group \(G_{3,2}\) should not be identified with \(SU(2,1)\) or
\(PU(2,1)\).  The latter groups are the unitary and projective-unitary
symmetry groups of the Hermitian \((2,1)\)-form itself.  By contrast,
\(G_{3,2}\) is the compact shared-determinant group associated to a
rank-three carrier and a rank-two unitary frame.  Its appearance here is
purely algebraic: the same determinant and phase-locking mechanisms used in
the projective calculus also produce the central quotient by \(\mu_6\).
\end{remark}

\begin{proposition}[Primitive scale weights]
Let \(N(\lambda)=\lambda\lambda^\ast\).  In the starred projective layer,
\[
  P(z,w)=h(z,w)h(w,z),\qquad
  D(z,w)=h(z,z)h(w,w),
\]
and
\[
  \Delta_H(z,w)=D(z,w)-P(z,w)
\]
all transform under \(z\mapsto\lambda z\), \(w\mapsto\mu w\) by the same
factor \(N(\lambda)N(\mu)\).  The triple product
\[
  T_H(z_1,z_2,z_3)
  =
  -h(z_1,z_2)h(z_2,z_3)h(z_3,z_1)
\]
transforms by \(N(\lambda_1)N(\lambda_2)N(\lambda_3)\).  The cross-ratio
numerator and denominator transform by the same endpoint scalar factor.
Thus the scale laws give three projective mechanisms:
\[
\begin{array}{ll}
\text{quotient} & P/D,\quad q_H=\Delta_H/D,\\
\text{norm-scalar phase class} & [T_H]\text{ modulo norm scalars},\\
\text{paired cancellation} & X(a,b;c,d).
\end{array}
\]
\end{proposition}

\begin{proof}
The identities follow directly from Hermitian scaling:
\[
  h(\lambda z,\mu w)=\lambda\mu^\ast h(z,w).
\]
For instance,
\[
  P(\lambda z,\mu w)
  =
  \lambda\mu^\ast h(z,w)\,\mu\lambda^\ast h(w,z)
  =
  N(\lambda)N(\mu)P(z,w),
\]
and the same factor multiplies \(D\), hence also \(\Delta_H=D-P\).  The
triple product contains each scalar once and each conjugate scalar once.  In
the cross-ratio expression, the four endpoint factors multiplying the
numerator are exactly the four endpoint factors multiplying the denominator.
\end{proof}

\begin{definition}[Reusable stratum-domain hypotheses]
Write
\[
\begin{array}{ll}
  \mathsf H_- & =\text{internal/non-null point stratum},\\
  \mathsf H_0 & =\text{null boundary stratum},\\
  \mathsf H_+ & =\text{external pole stratum}.
\end{array}
\]
Over \(\C\), these are \(\PH(V_-)\), \(\PH(V_0)\), and \(\PH(V_+)\).  Over a
field with involution, \(\mathsf H_0\) is the scale-stable null stratum and
\(\mathsf H_+\) is the chosen anisotropic pole stratum.  The reusable domain
predicates are:
\[
\begin{array}{ll}
\operatorname{NonNullPair}(A,B) & A,B\ne\mathsf H_0,\\
\operatorname{BoundaryPair}(A,B) & A=B=\mathsf H_0,\\
\operatorname{InteriorBoundary}(A,B) & \{A,B\}=\{\mathsf H_-,\mathsf H_0\},\\
\operatorname{InteriorPole}(A,B) & \{A,B\}=\{\mathsf H_-,\mathsf H_+\},\\
\operatorname{BoundaryPole}(A,B) & \{A,B\}=\{\mathsf H_0,\mathsf H_+\},\\
\operatorname{PolePair}(A,B) & A=B=\mathsf H_+.
\end{array}
\]
The symbol \(\mathsf H_-\) is analytic unless a model supplies an algebraic
non-null point stratum; the pole and boundary predicates are the field-uniform
ones used by the finite Hermitian theorem.
\end{definition}

The resulting interaction table is:
\begin{center}
\renewcommand{\arraystretch}{1.25}
\begin{tabular}{p{0.17\textwidth}p{0.35\textwidth}p{0.38\textwidth}}
Domain & Legal invariant & Geometric readout \\
\hline
\(\PH(V_-)\times\PH(V_-)\) &
\(\eta_H\), \(q_H\), and \(\Delta_H\) &
interior pair invariant; distance is an analytic readout of \(\eta_H\) \\
\(\PH(V_-)\times\PH(V_0)\) &
boundary pairing \(h(z,\xi)\) and normalized boundary kernels &
interior-to-boundary kernel; \(q_H\) is undefined because \(h(\xi,\xi)=0\) \\
\(\PH(V_-)\times\PH(V_+)\) &
non-null pair invariant and incidence \(h(z,p)=0\) &
position relative to the geodesic polar to \(p\) \\
\(\PH(V_0)\times\PH(V_0)\) &
boundary pairing, boundary cross-ratio, triple product, contact kernel &
CR/contact boundary calculus; distinct null points are nonorthogonal in
\(\CH^2\) \\
\(\PH(V_0)\times\PH(V_+)\) &
chain incidence \(h(\xi,p)=0\), chain kernels, chain cross-ratio readouts &
boundary point on or off the chain polar to \(p\) \\
\(\PH(V_+)\times\PH(V_+)\) &
polar spread and positive-pole pair invariant &
relative position of complex geodesics or boundary chains
\end{tabular}
\end{center}

The denominator rule is therefore simple: \(\eta_H\) and \(q_H\) are defined
only when both diagonal entries are nonzero.  As soon as a null point appears,
the calculus switches to boundary primitives: Hermitian pairings, cross
ratios, triple products, CR/contact density, and chain incidence.

\begin{theorem}[Domain decision and primitive routing]
For any pair of projective strata \(A,B\in\{\mathsf H_-,\mathsf H_0,\mathsf
H_+\}\), the non-null pair invariant \(\eta_H\), and hence \(q_H\), is legal
if and only if neither entry is \(\mathsf H_0\).  If at least one entry is
\(\mathsf H_0\), every legal readout must be one of the boundary primitives:
Hermitian boundary pairing, boundary cross-ratio, boundary triple product,
CR/contact density, interior-boundary kernel, or chain incidence.

More explicitly:
\[
\begin{array}{c|c}
\text{domain} & \text{primitive route}\\ \hline
\mathsf H_-\mathsf H_- & \eta_H,q_H,\Delta_H\\
\mathsf H_-\mathsf H_0 & h(z,\xi)\text{ and boundary/Poisson kernels}\\
\mathsf H_-\mathsf H_+ & \eta_H,q_H,\Delta_H\text{ plus }h(z,p)=0\\
\mathsf H_0\mathsf H_0 & h(\xi,\eta),X,T_H,\text{CR/contact data}\\
\mathsf H_0\mathsf H_+ & h(\xi,p)=0\text{ and chain readouts}\\
\mathsf H_+\mathsf H_+ & S_H,\eta_H,q_H,\Delta_H\text{ on poles.}
\end{array}
\]
\end{theorem}

\begin{proof}
The denominator of \(\eta_H\) is \(h(z,z)h(w,w)\).  It is nonzero exactly
when both entries are non-null.  A null entry makes this denominator vanish,
so the quotient defining \(\eta_H\) and \(q_H\) is not part of the legal
projective calculus.
The remaining entries in the table are polynomial or quotient expressions
whose stated hypotheses avoid the vanished null diagonal: raw Hermitian
pairings, triple products, cross-ratios with nonzero denominator, incidence
equations, contact forms in the analytic boundary chart, and pole-pair
invariants on non-null pole representatives.
\end{proof}

\begin{proposition}[Orthogonality readout by domain]
The equation \(h(u,v)=0\) has a domain-dependent interpretation:
\[
\begin{array}{c|c}
\text{domain} & \text{readout}\\ \hline
\mathsf H_-\mathsf H_- \text{ or non-null/non-null} & q_H=1\\
\mathsf H_-\mathsf H_+ & \text{the internal point lies on the polar geodesic}\\
\mathsf H_0\mathsf H_+ & \text{the boundary point lies on the polar chain}\\
\mathsf H_0\mathsf H_0 & \text{distinct null points are nonorthogonal in CH2}\\
\mathsf H_+\mathsf H_+ & \text{a pole-level relation between chains.}
\end{array}
\]
Thus orthogonality is not a single geometric sentence; it is an algebraic
equation routed through the appropriate stratum domain.
\end{proposition}

\begin{proof}
On the non-null pair domain, \(h(u,v)=0\) gives
\[
  \eta_H([u],[v])=0,\qquad q_H([u],[v])=1.
\]
For an internal point and a positive pole, \(h(u,p)=0\) is exactly the
definition of membership in the complex geodesic \(\PH(p^\perp\cap V_-)\).
For a null point and positive pole, the same equation is membership in the
boundary chain \(\PH(p^\perp\cap V_0)\).  Two distinct null lines in
\(\CH^2\) cannot be orthogonal: otherwise they would span a totally isotropic
two-plane, contradicting Witt index one for a \((2,1)\)-signature Hermitian
space.  The pole-pair case is the dual non-null pair calculus applied to
positive representatives.
\end{proof}

\section{Pair invariant and Hermitian quadrance}

\begin{definition}[Hermitian pair invariant]
For non-null projective points \([z]\) and \([w]\), define
\[
  \eta_H([z],[w])
  =
  \frac{|h(z,w)|^2}{h(z,z)h(w,w)}.
\]
The associated Hermitian quadrance is
\[
  q_H([z],[w])=1-\eta_H([z],[w]).
\]
\end{definition}

The quantity \(\eta_H\) is the standard Hermitian kernel appearing in the
Bergman distance formula.  Here it is repurposed as an algebraic projective
coordinate; the added calculus is the determinant complement \(q_H\), its
projective descent laws, and the null-join witnesses below.

\begin{proposition}[Scale invariance]
For nonzero \(\lambda,\mu\in\C\),
\[
  \eta_H([\lambda z],[\mu w])=\eta_H([z],[w]),
  \qquad
  q_H([\lambda z],[\mu w])=q_H([z],[w]).
\]
\end{proposition}

\begin{proof}
By Hermitian scaling,
\[
  h(\lambda z,\mu w)=\lambda\overline{\mu}\,h(z,w).
\]
The numerator is multiplied by \(|\lambda|^2|\mu|^2\), while the two diagonal
terms are multiplied by \(|\lambda|^2\) and \(|\mu|^2\).  The quotient is
therefore unchanged, and so is \(1-\eta_H\).
\end{proof}

\begin{proposition}[Gram determinant form]
Let
\[
  \Delta_H(z,w)
  =
  h(z,z)h(w,w)-|h(z,w)|^2.
\]
Then
\[
  q_H([z],[w])
  =
  \frac{\Delta_H(z,w)}{h(z,z)h(w,w)}.
\]
\end{proposition}

\begin{proof}
On the non-null pair domain \(h(z,z)h(w,w)\ne0\), so the defining quotient
for \(q_H\) can be put over a common denominator.  Since
\(|h(z,w)|^2=h(z,w)h(w,z)\), one obtains
\[
\begin{aligned}
q_H([z],[w])
&=1-\frac{|h(z,w)|^2}{h(z,z)h(w,w)}\\
&=\frac{h(z,z)h(w,w)-|h(z,w)|^2}{h(z,z)h(w,w)}
 =\frac{\Delta_H(z,w)}{h(z,z)h(w,w)}.
\end{aligned}
\]
\end{proof}

\begin{corollary}[Perpendicular value and null-join condition]
If \(h(z,w)=0\), then \(q_H([z],[w])=1\).  If \(z,w\) are non-null, then
\[
  q_H([z],[w])=0
  \quad\Longleftrightarrow\quad
  \Delta_H(z,w)=0.
\]
For distinct projective points this is the algebraic condition that the
Hermitian form restricted to the projective join has a null radical direction.
\end{corollary}

\begin{proof}
If \(h(z,w)=0\), then \(\eta_H([z],[w])=0\), so \(q_H([z],[w])=1\).  For
non-null \(z,w\), the denominator \(h(z,z)h(w,w)\) is nonzero, so the Gram
determinant formula gives
\[
  q_H([z],[w])=0
  \quad\Longleftrightarrow\quad
  \Delta_H(z,w)=0.
\]
The last statement is the usual determinant criterion: the restriction of
\(h\) to the span of \(z,w\) has Gram matrix
\[
  \begin{pmatrix}
  h(z,z) & h(z,w)\\
  h(w,z) & h(w,w)
  \end{pmatrix},
\]
whose determinant is \(\Delta_H(z,w)\).  If \([z]\ne[w]\), then \(z,w\)
are linearly independent and this span is two-dimensional.  Vanishing of the
determinant is exactly singularity of the restricted Gram matrix, so there is
a nonzero vector
\[
  x=\alpha z+\beta w
\]
in the radical of the restricted form.  In particular \(h(x,y)=0\) for every
\(y\in\operatorname{span}\{z,w\}\), and taking \(y=x\) gives
\(h(x,x)=0\).  Thus the projective join contains the null direction \([x]\).
Conversely, if the restricted form has a nonzero radical vector, then its
Gram matrix is singular, so the determinant \(\Delta_H(z,w)\) vanishes.
\end{proof}

\begin{proposition}[Null-join witness]
Let \(z,w\) be non-null representatives with \(a=h(z,z)\ne0\), and set
\[
  v=a\,w-h(w,z)\,z.
\]
Then
\[
  h(v,v)=a\,\Delta_H(z,w).
\]
Consequently, if \(\Delta_H(z,w)=0\), then \(v\) is null.  If moreover
\([z]\ne[w]\), then \(v\ne0\), so
\([v]\in\partial\CH^2\) lies on the projective join of \([z]\) and \([w]\).
Equivalently, on the non-null pair domain, a zero witness would force the two
standard Siegel representatives to have been projectively equivalent already.
\end{proposition}

\begin{proof}
Using Hermitian linearity in the first argument and conjugate-linearity in the
second, and writing \(a=h(z,z)\), one has
\[
\begin{aligned}
h(v,v)
&=h(a w-h(w,z)z,\,a w-h(w,z)z)\\
&=a\overline a\,h(w,w)
  -a\,h(z,w)h(w,z)
  -h(w,z)\overline a\,h(z,w)
  +h(w,z)h(z,w)h(z,z).
\end{aligned}
\]
The diagonal value \(a\) is real because \(h(z,z)=\overline{h(z,z)}\), hence
\(\overline a=a\).  Substituting \(h(z,z)=a\) and collecting the three
\(h(z,w)h(w,z)\) terms gives
\[
\begin{aligned}
h(v,v)
&=a^2h(w,w)-a\,h(z,w)h(w,z)\\
&=a\bigl(a\,h(w,w)-h(z,w)h(w,z)\bigr).
\end{aligned}
\]
Finally \(h(z,w)h(w,z)=|h(z,w)|^2\), so
\[
  h(v,v)
  =
  a\bigl(h(z,z)h(w,w)-|h(z,w)|^2\bigr)
  =
  a\,\Delta_H(z,w).
\]
If \(\Delta_H(z,w)=0\), this identity gives \(h(v,v)=0\), so any nonzero
\(v\) is a null vector.  Since \(v=a w-h(w,z)z\), it lies in
\(\operatorname{span}\{z,w\}\), hence its projective class lies on the
projective join of \([z]\) and \([w]\).  It remains only to check
nonvanishing for distinct projective points.  If \(v=0\), then
\(a w=h(w,z)z\), hence
\[
  w=\frac{h(w,z)}{a}\,z.
\]
Since \(w\) is non-null, the scalar \(h(w,z)/a\) cannot be zero.  Thus
\([w]=[z]\).  The contrapositive gives the stated nonvanishing for distinct
projective points.
\end{proof}

\section{Polar spread and complex geodesics}

Positive polar vectors and their complex geodesics are standard
complex-hyperbolic geometry.  The point here is to reuse the same determinant
pair calculus on the dual positive-pole data, producing a spread-like
algebraic readout rather than a new class of geodesics.

\begin{definition}[Polar hyperplane and complex geodesic]
For a positive vector \(p\in V_+\), define
\[
  p^\perp=\{z\in V:h(z,p)=0\}.
\]
The associated complex geodesic and its boundary chain are
\[
  L_p=\PH(p^\perp\cap V_-),\qquad
  \mathcal C_p=\PH(p^\perp\cap V_0).
\]
\end{definition}

\begin{proposition}[Projective polarity of chains]
If \(\lambda\in\C^\times\), then
\[
  (\lambda p)^\perp=p^\perp,\qquad
  L_{\lambda p}=L_p,\qquad
  \mathcal C_{\lambda p}=\mathcal C_p.
\]
Equivalently, the incidence relation
\[
  [x]\in\mathcal C_p
  \quad\Longleftrightarrow\quad
  h(x,x)=0,\ h(x,p)=0
\]
depends only on the projective classes \([x]\) and \([p]\).
\end{proposition}

\begin{proof}
For every \(z\in V\),
\[
  h(z,\lambda p)=\overline{\lambda}\,h(z,p).
\]
Since \(\lambda\ne0\), the vanishing of \(h(z,\lambda p)\) is equivalent to
the vanishing of \(h(z,p)\).  Thus the orthogonal hyperplane, its negative
part, and its null part are unchanged.  If \(x\) is also rescaled by
\(\mu\in\C^\times\), then
\[
  h(\mu x,p)=\mu h(x,p),
  \qquad
  h(\mu x,\mu x)=|\mu|^2h(x,x),
\]
so nullity and incidence are unchanged on the projective class of \(x\).
\end{proof}

\begin{definition}[Hermitian spread]
For positive polar vectors \(p,q\), define
\[
  S_H(p,q)
  =
  1-\frac{|h(p,q)|^2}{h(p,p)h(q,q)}.
\]
\end{definition}

This is the polar-vector version of the pair determinant invariant.  It is a
natural Hermitian analogue of spread between hyperplanes, but it is not enough
to describe \(\CH^2\) by itself: triple phase and boundary contact data remain
essential.

At quotient level, a pair of positive polar representatives is simply a
positive-pole pair.  Since positive representatives have nonzero Hermitian
diagonal, every such pair lies in the non-null pair domain.  The projective
polar spread is therefore exactly the already-defined projective quadrance of
the associated non-null polar pair:
\[
  S_H([p],[q])
  =
  q_H([p],[q]).
\]
This is the precise Hermitian analogue of a Wildberger-style spread identity:
it is not a new metric primitive, but the same determinant quotient applied to
dual polar data.

\begin{proposition}[Geodesic position discriminator]
Let \(p,q\in V_+\) be independent positive poles and set
\[
  \Delta_H(p,q)=h(p,p)h(q,q)-|h(p,q)|^2.
\]
Then:
\[
\begin{array}{c|c}
\Delta_H(p,q)>0 & L_p\cap L_q\text{ consists of one interior point}\\
\Delta_H(p,q)=0 & \mathcal C_p\cap\mathcal C_q\text{ consists of one boundary point}\\
\Delta_H(p,q)<0 & \overline{L_p}\cap\overline{L_q}=\varnothing
\end{array}
\]
where \(\overline{L_p}=L_p\cup\mathcal C_p\).  Equivalently, since
\[
  \Delta_H(p,q)=h(p,p)h(q,q)S_H(p,q)
\]
and the diagonal factors are positive, the sign of the polar spread is the
sign of the intersection type.  The special case \(h(p,q)=0\), equivalently
\(S_H(p,q)=1\), is an orthogonal interior intersection.
\end{proposition}

\begin{proof}
The intersection of the two projective closures is represented by the
one-dimensional orthogonal complement
\[
  \Span\{p,q\}^{\perp}.
\]
The restriction of \(h\) to \(\Span\{p,q\}\) has Hermitian Gram matrix
\[
  \begin{pmatrix}
    h(p,p) & h(p,q)\\
    h(q,p) & h(q,q)
  \end{pmatrix}
\]
with determinant \(\Delta_H(p,q)\).  Since \(h(p,p)>0\), this two-plane is
positive definite when \(\Delta_H(p,q)>0\), degenerate when
\(\Delta_H(p,q)=0\), and has signature \((1,1)\) when
\(\Delta_H(p,q)<0\).  The ambient signature is \((2,1)\), so the orthogonal
complement is respectively negative, null, or positive.  These cases give an
interior point, a boundary point, or no point in the closed ball.  The
identity with \(S_H\) is the determinant form of polar spread.  If
\(h(p,q)=0\), then the pole two-plane is positive orthogonal, and the two
complex geodesics meet orthogonally at the negative complement point.
\end{proof}

\begin{proposition}[Unique chain through two boundary points]
Any two distinct boundary points \([\xi]\ne[\eta]\) lie on a unique chain.
Its pole is the positive projective line
\[
  [p]=\PH\bigl(\Span\{\xi,\eta\}^{\perp}\bigr).
\]
\end{proposition}

\begin{proof}
Distinct null lines cannot be Hermitian-orthogonal in signature \((2,1)\).
Otherwise \(\Span\{\xi,\eta\}\) would be a two-dimensional totally isotropic
subspace, contradicting Witt index one.  Hence
\[
  h(\xi,\eta)\ne0.
\]
The restriction of \(h\) to \(\Span\{\xi,\eta\}\) has Gram matrix
\[
  \begin{pmatrix}
    0 & h(\xi,\eta)\\
    h(\eta,\xi) & 0
  \end{pmatrix},
\]
of determinant \(-|h(\xi,\eta)|^2\).  Thus the span has signature \((1,1)\),
so its orthogonal complement is a positive line.  Let \(p\) represent that
positive line.  Then \(h(\xi,p)=h(\eta,p)=0\), so both boundary points lie on
\(\mathcal C_p\).  If another positive pole \(p'\) defines a chain through
both points, then \(p'\in\Span\{\xi,\eta\}^{\perp}\), so \([p']=[p]\).
\end{proof}

\section{Triple product}

\begin{definition}[Hermitian triple product]
For projective representatives \(z_1,z_2,z_3\), define
\[
  T_H(z_1,z_2,z_3)
  =
  -h(z_1,z_2)h(z_2,z_3)h(z_3,z_1).
\]
Its phase class is the algebraic origin of the Cartan angular invariant.
\end{definition}

\begin{proposition}[Triple-product scaling]
If \(z_i\mapsto \lambda_i z_i\), then
\[
  T_H(\lambda_1z_1,\lambda_2z_2,\lambda_3z_3)
  =
  |\lambda_1\lambda_2\lambda_3|^2\,T_H(z_1,z_2,z_3).
\]
Thus the class of \(T_H\) modulo positive real norm factors is projectively
well-defined.
\end{proposition}

\begin{proof}
By Hermitian scaling,
\[
  h(\lambda_i z_i,\lambda_j z_j)
  =
  \lambda_i\overline{\lambda_j}\,h(z_i,z_j).
\]
Therefore
\[
\begin{aligned}
T_H(\lambda_1z_1,\lambda_2z_2,\lambda_3z_3)
&=-(\lambda_1\overline{\lambda_2}h(z_1,z_2))
  (\lambda_2\overline{\lambda_3}h(z_2,z_3))
  (\lambda_3\overline{\lambda_1}h(z_3,z_1))\\
&=-(\lambda_1\overline{\lambda_1})
   (\lambda_2\overline{\lambda_2})
   (\lambda_3\overline{\lambda_3})
   h(z_1,z_2)h(z_2,z_3)h(z_3,z_1)\\
&=|\lambda_1\lambda_2\lambda_3|^2\,T_H(z_1,z_2,z_3).
\end{aligned}
\]
The scalar factor is a positive real norm factor, so the phase class and the
class modulo positive real norm factors are projectively well-defined.
\end{proof}

\begin{proposition}[Cyclic and reverse behavior]
The triple product is cyclically invariant:
\[
  T_H(z_1,z_2,z_3)=T_H(z_2,z_3,z_1).
\]
Reversing cyclic orientation conjugates it:
\[
  T_H(z_1,z_3,z_2)=\overline{T_H(z_1,z_2,z_3)}.
\]
\end{proposition}

\begin{proof}
Cyclic invariance follows by commutativity of complex multiplication:
\[
  h(z_1,z_2)h(z_2,z_3)h(z_3,z_1)
  =
  h(z_2,z_3)h(z_3,z_1)h(z_1,z_2).
\]
For reversal, Hermitian symmetry gives \(h(z_j,z_i)=\overline{h(z_i,z_j)}\).
Thus
\[
\begin{aligned}
T_H(z_1,z_3,z_2)
&=-h(z_1,z_3)h(z_3,z_2)h(z_2,z_1)\\
&=-\overline{h(z_3,z_1)}\,
    \overline{h(z_2,z_3)}\,
    \overline{h(z_1,z_2)}\\
&=-\overline{h(z_3,z_1)h(z_2,z_3)h(z_1,z_2)}\\
&=-\overline{h(z_1,z_2)h(z_2,z_3)h(z_3,z_1)}\\
&=\overline{-h(z_1,z_2)h(z_2,z_3)h(z_3,z_1)}
 =\overline{T_H(z_1,z_2,z_3)}.
\end{aligned}
\]
\end{proof}

\section{Analytic recovery of classical invariants}

The algebraic quantities above contain the usual analytic invariants of
complex hyperbolic geometry.  The point of the calculus is not to discard
distance and angle, but to postpone the transcendental readout until after
the projective Hermitian identities have been established.

\begin{proposition}[Bergman distance from the pair invariant]
For negative vectors \(z,w\in V_-\), the complex-hyperbolic distance satisfies
\[
  \cosh^2\!\left(\frac{d([z],[w])}{2}\right)
  =
  \frac{|h(z,w)|^2}{h(z,z)h(w,w)}
  =
  \eta_H([z],[w]).
\]
Equivalently,
\[
  d([z],[w])=2\,\operatorname{arcosh}\sqrt{\eta_H([z],[w])}
  =
  2\,\operatorname{arcosh}\sqrt{1-q_H([z],[w])}.
\]
\end{proposition}

\begin{proof}
This is the standard Bergman distance formula in the projective ball model.
Both diagonal terms are negative on \(V_-\), so their product is positive and
the displayed quotient is scale-invariant.  The second formula is obtained by
solving the first for \(d\) and using \(q_H=1-\eta_H\).
\end{proof}

\begin{proposition}[Cartan angular invariant from triple phase]
For a triple \(z_1,z_2,z_3\) for which \(T_H(z_1,z_2,z_3)\ne0\), define
\[
  \mathbb A(z_1,z_2,z_3)
  =
  \argg T_H(z_1,z_2,z_3)
  =
  \argg\!\left(
    -h(z_1,z_2)h(z_2,z_3)h(z_3,z_1)
  \right).
\]
For boundary triples this is the usual Cartan angular invariant, with the
standard range convention \(\mathbb A\in[-\pi/2,\pi/2]\).  If
\[
  \Omega_{123}=\frac{T_H(z_1,z_2,z_3)}{|T_H(z_1,z_2,z_3)|},
\]
then
\[
  \Omega_{123}=e^{i\mathbb A(z_1,z_2,z_3)},
  \qquad
  \Omega_{123}+\overline{\Omega}_{123}=2\cos\mathbb A(z_1,z_2,z_3).
\]
\end{proposition}

\begin{proof}
The triple product changes under projective rescaling by the positive real
factor \(|\lambda_1\lambda_2\lambda_3|^2\).  Its argument is therefore
projectively well-defined.  The displayed expression is exactly Cartan's
triple-product formula, and the final identities are the polar form of a
nonzero complex number.
\end{proof}

\section{Pair and triple cycles under projective-unitary symmetry}

The pair and triple invariants can also be organized as Hermitian cycles.
This emphasizes the separation between metric modulus and complex phase under
the projective-unitary symmetries of the \((2,1)\)-form.

\begin{definition}[Hermitian pair and triple cycles]
For non-null representatives \(z_i\), write
\[
  a_i=h(z_i,z_i),\qquad h_{ij}=h(z_i,z_j).
\]
The pair cycle is
\[
  C_{2,ij}
  =
  \frac{h_{ij}h_{ji}}{a_i a_j}
  =
  \frac{|h_{ij}|^2}{a_i a_j}
  =
  \eta_H([z_i],[z_j]).
\]
For a non-null triple, the normalized triple cycle is
\[
  C_{3,123}
  =
  \frac{h_{12}h_{23}h_{31}}{a_1a_2a_3}.
\]
Thus \(C_{3,123}\) is the same normalized product denoted
\(\Omega_{123}\) in the triangle identities below, while the signed Cartan
convention is
\[
  \frac{T_H(z_1,z_2,z_3)}{a_1a_2a_3}=-C_{3,123}.
\]
\end{definition}

\begin{proposition}[Projective-unitary invariance of cycles]
Let \(g\in U(2,1)\); in particular, let \(g\in SU(2,1)\).  Then
\[
  C_{2,ij}(gz_i,gz_j)=C_{2,ij}(z_i,z_j),
  \qquad
  C_{3,123}(gz_1,gz_2,gz_3)=C_{3,123}(z_1,z_2,z_3).
\]
The same formulas descend to \(PU(2,1)\) on projective classes.
\end{proposition}

\begin{proof}
The defining property of \(U(2,1)\) is
\[
  h(gu,gv)=h(u,v).
\]
Hence every numerator factor \(h_{ij}\) and every diagonal factor \(a_i\) in
the definitions of \(C_2\) and \(C_3\) is unchanged.  This proves invariance
under \(U(2,1)\), hence under the subgroup \(SU(2,1)\).  The determinant-one
condition is not used for these cycle invariants; preservation of \(h\) is
the essential input.  Projective descent follows from the scale cancellations
already built into the definitions:
under \(z_i\mapsto\lambda_i z_i\), the numerator and denominator of
\(C_{2,ij}\) are both multiplied by
\(|\lambda_i|^2|\lambda_j|^2\), and the numerator and denominator of
\(C_{3,123}\) are both multiplied by
\(|\lambda_1\lambda_2\lambda_3|^2\).
\end{proof}

\begin{proposition}[Stratified meaning of the pair cycle]
The invariant \(C_{2,ij}\) is the unoriented pair modulus, but its geometric
readout depends on the projective stratum.
\begin{enumerate}
\item If \(z_i,z_j\in V_-\), then
\[
  C_{2,ij}\ge1,\qquad
  \cosh^2\!\left(\frac{d([z_i],[z_j])}{2}\right)=C_{2,ij}.
\]
\item If \(p,q\in V_+\) are positive poles, then
\[
  S_H([p],[q])=1-C_{2,pq}
\]
is the polar spread of the associated complex geodesics.
\item If \(\xi,\eta\in V_0\), then the denominator of \(C_2\) vanishes.
Moreover, if \([\xi]\ne[\eta]\), then \(h(\xi,\eta)\ne0\).
\end{enumerate}
\end{proposition}

\begin{proof}
The first statement is the Bergman distance formula, and the inequality
\(C_{2,ij}\ge1\) follows because \(\cosh^2(d/2)\ge1\).  The second statement
is the definition of polar spread as the same determinant quotient applied to
positive-pole data:
\[
  S_H([p],[q])
  =
  1-\frac{|h(p,q)|^2}{h(p,p)h(q,q)}.
\]
For the null case, \(h(\xi,\xi)h(\eta,\eta)=0\), so the pair-cycle quotient
is not defined.  If distinct null lines were orthogonal, then
\(\Span\{\xi,\eta\}\) would be a two-dimensional totally isotropic subspace,
contradicting Witt index one for a Hermitian form of signature \((2,1)\).
\end{proof}

\begin{theorem}[Modulus-phase decoupling for the triple cycle]
On the non-null triple domain,
\[
  C_{3,123}\overline{C_{3,123}}
  =
  C_{2,12}C_{2,23}C_{2,31}.
\]
Consequently the modulus of \(C_{3,123}\) is determined by the three pair
cycles.  When \(C_{3,123}\ne0\), its independent contribution is the phase
\[
  \argg C_{3,123}.
\]
With the signed convention \(T_H=-h_{12}h_{23}h_{31}\), the normalized Cartan
phase is
\[
  \frac{T_H}{|T_H|}
  =
  -\operatorname{sgn}(a_1a_2a_3)\,
  \frac{C_{3,123}}{|C_{3,123}|}
\]
whenever the triple product is nonzero.  Thus the Cartan readout differs from
\(\argg C_{3,123}\) only by the fixed real sign determined by the non-null
strata of the three vertices.
\end{theorem}

\begin{proof}
Using \(a_i=\overline{a_i}\), one computes
\[
\begin{aligned}
C_{3,123}\overline{C_{3,123}}
&=
\frac{h_{12}h_{23}h_{31}}{a_1a_2a_3}
\frac{h_{21}h_{32}h_{13}}{a_1a_2a_3} \\
&=
\frac{h_{12}h_{21}}{a_1a_2}
\frac{h_{23}h_{32}}{a_2a_3}
\frac{h_{31}h_{13}}{a_3a_1}  \\
&=
C_{2,12}C_{2,23}C_{2,31}.
\end{aligned}
\]
Thus the absolute value of \(C_{3,123}\) is fixed by the \(C_2\)-data.  The
remaining datum in \(C_{3,123}\) is its argument.  Since
\[
  T_H=-h_{12}h_{23}h_{31}=-(a_1a_2a_3)C_{3,123}
\]
and \(a_1a_2a_3\) is a nonzero real number on the non-null triple domain, the
displayed normalized phase relation follows.
\end{proof}

\begin{definition}[Edge phases]
On a non-null triple with nonzero pairings \(h_{12}h_{23}h_{31}\ne0\), define
the oriented edge phases
\[
  \omega_{ij}
  =
  \frac{h_{ij}}{|h_{ij}|}
  \in U(1).
\]
These are defined for oriented edges: by Hermitian symmetry,
\[
  \omega_{ji}=\overline{\omega_{ij}}.
\]
\end{definition}

\begin{theorem}[Edge phases and cyclic holonomy]
The individual edge phases \(\omega_{ij}\) are not projective invariants, but
their cyclic product is.  More precisely, if
\[
  z_i\longmapsto \lambda_i z_i,
  \qquad
  u_i=\frac{\lambda_i}{|\lambda_i|}\in U(1),
\]
then
\[
  \omega_{ij}\longmapsto u_i\overline{u_j}\,\omega_{ij}.
\]
Consequently
\[
  \omega_{12}\omega_{23}\omega_{31}
  \longmapsto
  \omega_{12}\omega_{23}\omega_{31}.
\]
Moreover
\[
  \frac{C_{3,123}}{|C_{3,123}|}
  =
  \operatorname{sgn}(a_1a_2a_3)\,
  \omega_{12}\omega_{23}\omega_{31},
\]
and, with \(T_H=-h_{12}h_{23}h_{31}\),
\[
  \frac{T_H}{|T_H|}
  =
  -\omega_{12}\omega_{23}\omega_{31}.
\]
\end{theorem}

\begin{proof}
Under \(z_i\mapsto\lambda_i z_i\), Hermitian scaling gives
\[
  h_{ij}\longmapsto \lambda_i\overline{\lambda_j}\,h_{ij}.
\]
Taking phases gives
\[
  \frac{\lambda_i\overline{\lambda_j}h_{ij}}
       {|\lambda_i\overline{\lambda_j}h_{ij}|}
  =
  \frac{\lambda_i}{|\lambda_i|}
  \frac{\overline{\lambda_j}}{|\lambda_j|}
  \frac{h_{ij}}{|h_{ij}|}
  =
  u_i\overline{u_j}\,\omega_{ij}.
\]
Multiplying around the oriented cycle,
\[
\begin{aligned}
\omega_{12}\omega_{23}\omega_{31}
&\longmapsto
(u_1\overline{u_2})(u_2\overline{u_3})(u_3\overline{u_1})
\omega_{12}\omega_{23}\omega_{31}\\
&=\omega_{12}\omega_{23}\omega_{31}.
\end{aligned}
\]
Thus the cyclic edge phase is projectively invariant even though the
individual edge phases depend on the chosen projective representatives.

For the relation with \(C_3\), write
\[
  h_{12}h_{23}h_{31}
  =
  |h_{12}|\,|h_{23}|\,|h_{31}|\,
  \omega_{12}\omega_{23}\omega_{31}.
\]
Since \(a_1a_2a_3\) is real and nonzero,
\[
  C_{3,123}
  =
  \frac{|h_{12}|\,|h_{23}|\,|h_{31}|}{|a_1a_2a_3|}
  \operatorname{sgn}(a_1a_2a_3)\,
  \omega_{12}\omega_{23}\omega_{31}.
\]
Dividing by \(|C_{3,123}|\) gives the displayed formula.  Finally
\[
  T_H=-h_{12}h_{23}h_{31},
\]
so its normalized phase is the negative of the cyclic edge phase.
\end{proof}

\begin{remark}[Projective frame interpretation]
The phases \(\omega_{ij}\) form a \(U(1)\)-valued edge cochain on the
oriented triangle: changing projective representatives applies the vertex
rescaling factors \(u_i\overline{u_j}\).  The product
\[
  \omega_{12}\omega_{23}\omega_{31}
\]
is the corresponding projective cyclic phase.  Thus the decoupling can be
read as:
\[
\begin{array}{c|c}
C_{2,ij} & \text{edge moduli}\\
\omega_{ij} & \text{representative-dependent edge phases}\\
\omega_{12}\omega_{23}\omega_{31} & \text{projectively invariant cyclic phase}\\
C_{3,123} & \text{modulus fixed by \(C_2\), phase fixed by the cyclic product.}
\end{array}
\]
\end{remark}

\begin{theorem}[Gram compatibility of pair modulus and triple phase]
Let
\[
  D_{123}=\det(h(z_i,z_j))_{1\le i,j\le3}.
\]
Then
\[
  \frac{D_{123}}{a_1a_2a_3}
  =
  1-C_{2,12}-C_{2,23}-C_{2,31}
  +C_{3,123}+\overline{C_{3,123}}.
\]
Equivalently, if
\[
  C_{3,123}=\rho e^{i\alpha},
  \qquad
  \rho^2=C_{2,12}C_{2,23}C_{2,31},
\]
then
\[
  \frac{D_{123}}{a_1a_2a_3}
  =
  1-C_{2,12}-C_{2,23}-C_{2,31}
  +2\rho\cos\alpha.
\]
\end{theorem}

\begin{proof}
This is the three-point Hermitian Gram identity divided by \(a_1a_2a_3\).
Indeed,
\[
  \frac{|h_{ij}|^2}{a_i a_j}=C_{2,ij},
  \qquad
  \frac{h_{12}h_{23}h_{31}}{a_1a_2a_3}=C_{3,123}.
\]
Substitution gives the first displayed identity.  The second follows by
writing \(C_{3,123}=\rho e^{i\alpha}\), so that
\[
  C_{3,123}+\overline{C_{3,123}}=2\rho\cos\alpha,
\]
and using the modulus identity from the preceding theorem.
\end{proof}

\begin{remark}[Boundary triples]
For boundary triples the diagonal denominators vanish, so the normalized
cycles \(C_2\) and \(C_3\) are not the correct objects.  The boundary phase is
carried by the raw triple product
\[
  T_\partial(\xi_1,\xi_2,\xi_3)
  =
  -h(\xi_1,\xi_2)h(\xi_2,\xi_3)h(\xi_3,\xi_1).
\]
Its argument is Cartan's boundary angular invariant.  In the standard
boundary theory, \(\mathbb A=0\) characterizes triples lying on a totally real
circle, while \(|\mathbb A|=\pi/2\) characterizes triples lying on a chain.
\end{remark}

\section{Three-point Gram identity}

Let \(h_{ij}=h(z_i,z_j)\).  The Hermitian Gram determinant is
\[
  \det(h_{ij})_{1\le i,j\le3}.
\]

\begin{theorem}[Three-point Hermitian Gram identity]
For \(a=h(z_1,z_1)\), \(b=h(z_2,z_2)\), \(c=h(z_3,z_3)\), and
\[
  \tau=h(z_1,z_2)h(z_2,z_3)h(z_3,z_1),
\]
one has
\[
\begin{aligned}
\det(h_{ij})
&=abc
-c|h(z_1,z_2)|^2
-a|h(z_2,z_3)|^2
-b|h(z_3,z_1)|^2  \\
&\qquad +\tau+\overline{\tau}.
\end{aligned}
\]
\end{theorem}

\begin{proof}
Expand the determinant of the Hermitian Gram matrix along the first row:
\[
\begin{aligned}
\det(h_{ij})
&=h_{11}(h_{22}h_{33}-h_{23}h_{32})
 -h_{12}(h_{21}h_{33}-h_{23}h_{31})\\
&\qquad
 +h_{13}(h_{21}h_{32}-h_{22}h_{31}).
\end{aligned}
\]
Substitute \(h_{11}=a\), \(h_{22}=b\), \(h_{33}=c\), and use Hermitian
symmetry \(h_{ji}=\overline{h_{ij}}\).  This gives
\[
\begin{aligned}
\det(h_{ij})
&=abc-a\,h_{23}h_{32}
  -h_{12}h_{21}c+h_{12}h_{23}h_{31}\\
&\qquad
  +h_{13}h_{21}h_{32}-h_{13}b h_{31}.
\end{aligned}
\]
The three pair-modulus terms are
\[
  h_{23}h_{32}=|h_{23}|^2,\qquad
  h_{12}h_{21}=|h_{12}|^2,\qquad
  h_{13}h_{31}=|h_{31}|^2,
\]
where the last equality uses \(h_{13}=\overline{h_{31}}\).  The two remaining
terms are the two orientations of the triple product without the conventional
minus sign:
\[
  h_{12}h_{23}h_{31}=\tau,
  \qquad
  h_{13}h_{21}h_{32}
  =\overline{h_{31}}\overline{h_{12}}\overline{h_{23}}
  =\overline{\tau}.
\]
Collecting these six terms yields
\[
\begin{aligned}
\det(h_{ij})
&=abc
-c|h_{12}|^2
-a|h_{23}|^2
-b|h_{31}|^2
+\tau+\overline{\tau},
\end{aligned}
\]
which is the displayed identity.
\end{proof}

This is the first place where \(\CH^2\) visibly departs from a two-invariant
quadrance/spread story.  Pair moduli do not determine the three-point
configuration; the real part of the triple product is also present, and the
phase class records orientation-sensitive information.

\section{Triangle trigonometry}

The preceding Gram identity is already the raw algebraic triangle law.  This
section rewrites it in the language of quadrance, phase, and polar spread, so
that the analogy with universal trigonometry is explicit.

\begin{definition}[Normalized triangle invariants]
For non-null representatives \(z_1,z_2,z_3\), write
\[
  a_i=h(z_i,z_i),\qquad h_{ij}=h(z_i,z_j),
\]
\[
  \Delta_{ij}=a_i a_j-|h_{ij}|^2,
  \qquad
  q_{ij}=\frac{\Delta_{ij}}{a_i a_j}.
\]
Set
\[
  D_{123}=\det(h_{ij})_{1\le i,j\le3},
  \qquad
  \Omega_{123}=\frac{h_{12}h_{23}h_{31}}{a_1a_2a_3},
\]
and let
\[
  \Theta_{123}=\Omega_{123}+\overline{\Omega}_{123}.
\]
\end{definition}

\begin{theorem}[Hermitian triangle constraint]
On the non-null triple domain,
\[
  \frac{D_{123}}{a_1a_2a_3}
  =
  q_{12}+q_{23}+q_{31}-2+\Theta_{123}.
\]
Moreover,
\[
  \Omega_{123}\overline{\Omega}_{123}
  =
  (1-q_{12})(1-q_{23})(1-q_{31}).
\]
\end{theorem}

\begin{proof}
Divide the three-point Gram identity by \(a_1a_2a_3\).  Since
\[
  \frac{|h_{ij}|^2}{a_ia_j}=1-q_{ij},
\]
the pair-modulus terms give
\[
\begin{aligned}
1-(1-q_{12})-(1-q_{23})-(1-q_{31})
&=q_{12}+q_{23}+q_{31}-2.
\end{aligned}
\]
The remaining two terms divide to
\[
  \frac{h_{12}h_{23}h_{31}}{a_1a_2a_3}
  +
  \frac{\overline{h_{12}h_{23}h_{31}}}{a_1a_2a_3}
  =
  \Omega_{123}+\overline{\Omega}_{123},
\]
because each \(a_i\) is real.  This proves the first identity.

For the modulus identity,
\[
\begin{aligned}
\Omega_{123}\overline{\Omega}_{123}
&=
\frac{|h_{12}|^2|h_{23}|^2|h_{31}|^2}
     {(a_1a_2a_3)^2}  \\
&=
\frac{|h_{12}|^2}{a_1a_2}\,
\frac{|h_{23}|^2}{a_2a_3}\,
\frac{|h_{31}|^2}{a_3a_1},
\end{aligned}
\]
where the last displayed line denotes the product of the three displayed
factors.  Each factor is \(1-q_{ij}\), giving the stated formula.
\end{proof}

\begin{definition}[Vertex polar spread]
Assume \(z_1,z_2,z_3\) are linearly independent and that the side spans
\(\langle z_1,z_2\rangle\) and \(\langle z_2,z_3\rangle\) are
nondegenerate.  Let \(p_3\) be a nonzero polar vector to
\(\langle z_1,z_2\rangle\), and let \(p_1\) be a nonzero polar vector to
\(\langle z_2,z_3\rangle\):
\[
  h(p_3,z_1)=h(p_3,z_2)=0,\qquad
  h(p_1,z_2)=h(p_1,z_3)=0.
\]
The vertex spread at \(z_2\) is
\[
  s_2=q_H([p_3],[p_1]).
\]
\end{definition}

\begin{proposition}[Vertex spread cofactor formula]
For a nondegenerate triangle whose side polars \(p_3,p_1\) are non-null,
\[
  1-s_2
  =
  \frac{|h_{12}h_{23}-a_2h_{13}|^2}{\Delta_{12}\Delta_{23}}.
\]
\end{proposition}

\begin{proof}
Let \(G=(h_{ij})\) be the three-by-three Hermitian Gram matrix.  The polar
vector to the side opposite \(z_i\) is computed, up to projective scale, by
the \(i\)-th cofactor column of \(G\).  Equivalently, Cramer's rule gives the
Gram matrix of the three side polars, up to one common nonzero scalar, as the
cofactor matrix of \(G\).  Hence
\[
  h(p_3,p_3)\sim \Delta_{12},\qquad
  h(p_1,p_1)\sim \Delta_{23},
\]
and
\[
  h(p_3,p_1)\sim h_{12}h_{23}-a_2h_{13}.
\]
The common scalar and the independent projective scalings of \(p_3,p_1\)
cancel in the projective quadrance \(q_H([p_3],[p_1])\).  Therefore
\[
\begin{aligned}
s_2
&=1-\frac{|h(p_3,p_1)|^2}{h(p_3,p_3)h(p_1,p_1)}\\
&=1-\frac{|h_{12}h_{23}-a_2h_{13}|^2}
          {\Delta_{12}\Delta_{23}},
\end{aligned}
\]
which is the claim.
\end{proof}

\begin{theorem}[Hermitian cross law]
On the domain where the vertex spread \(s_2\) is defined,
\[
  q_{12}q_{23}s_2
  =
  q_{12}+q_{23}+q_{31}-2+\Theta_{123}.
\]
Equivalently, in denominator-cleared form,
\[
  \Delta_{12}\Delta_{23}s_2
  =
  a_2D_{123}.
\]
\end{theorem}

\begin{proof}
By the cofactor formula,
\[
  s_2
  =
  \frac{\Delta_{12}\Delta_{23}
        -|h_{12}h_{23}-a_2h_{13}|^2}
       {\Delta_{12}\Delta_{23}}.
\]
Expand the squared modulus:
\[
\begin{aligned}
|h_{12}h_{23}-a_2h_{13}|^2
&=|h_{12}|^2|h_{23}|^2
  +a_2^2|h_{31}|^2
  -a_2(\tau+\overline{\tau}),
\end{aligned}
\]
where \(\tau=h_{12}h_{23}h_{31}\) and \(h_{13}=\overline{h_{31}}\).  Thus
\[
\begin{aligned}
\Delta_{12}\Delta_{23}
&-|h_{12}h_{23}-a_2h_{13}|^2\\
&=(a_1a_2-|h_{12}|^2)(a_2a_3-|h_{23}|^2)\\
&\qquad
  -|h_{12}|^2|h_{23}|^2
  -a_2^2|h_{31}|^2
  +a_2(\tau+\overline{\tau})\\
&=a_2\bigl(
    a_1a_2a_3
    -a_3|h_{12}|^2
    -a_1|h_{23}|^2
    -a_2|h_{31}|^2
    +\tau+\overline{\tau}
  \bigr)\\
&=a_2D_{123}.
\end{aligned}
\]
This proves the denominator-cleared identity.  Dividing by
\(a_1a_2^2a_3\) gives
\[
  q_{12}q_{23}s_2
  =
  \frac{D_{123}}{a_1a_2a_3}.
\]
The normalized form follows from the Hermitian triangle constraint.
\end{proof}

\begin{corollary}[Hermitian Pythagoras]
If the two side polars at \(z_2\) are Hermitian-orthogonal, equivalently
\(s_2=1\), then
\[
  q_{12}q_{23}
  =
  q_{12}+q_{23}+q_{31}-2+\Theta_{123}.
\]
Equivalently,
\[
  \Delta_{12}\Delta_{23}=a_2D_{123}.
\]
\end{corollary}

\begin{proof}
The condition \(s_2=1\) is exactly \(h(p_3,p_1)=0\) on the non-null polar
domain, since \(s_2=q_H([p_3],[p_1])\).  Substituting \(s_2=1\) into the
Hermitian cross law gives both displayed identities.
\end{proof}

\begin{definition}[Cyclic side polars and spreads]
For a linearly independent non-null triangle \(z_1,z_2,z_3\) whose side spans
are nondegenerate, choose side polars \(p_i\) by
\[
  h(p_i,z_j)=h(p_i,z_k)=0,
  \qquad
  \{i,j,k\}=\{1,2,3\}.
\]
Thus \(p_i\) is polar to the side opposite \(z_i\).  When the relevant side
polars are non-null, define the three vertex spreads by
\[
  s_1=q_H([p_2],[p_3]),\qquad
  s_2=q_H([p_3],[p_1]),\qquad
  s_3=q_H([p_1],[p_2]).
\]
Define the three spread numerators
\[
\begin{aligned}
\Sigma_1&=\Delta_{31}\Delta_{12}
          -|h_{31}h_{12}-a_1h_{32}|^2,\\
\Sigma_2&=\Delta_{12}\Delta_{23}
          -|h_{12}h_{23}-a_2h_{13}|^2,\\
\Sigma_3&=\Delta_{23}\Delta_{31}
          -|h_{23}h_{31}-a_3h_{21}|^2.
\end{aligned}
\]
\end{definition}

\begin{theorem}[Cyclic Hermitian spread law]
On the above cyclic nondegenerate domain,
\[
  \Sigma_i=a_iD_{123},
  \qquad i=1,2,3.
\]
Equivalently, whenever the displayed spreads are defined,
\[
\begin{aligned}
\Delta_{31}\Delta_{12}s_1&=a_1D_{123},\\
\Delta_{12}\Delta_{23}s_2&=a_2D_{123},\\
\Delta_{23}\Delta_{31}s_3&=a_3D_{123}.
\end{aligned}
\]
In normalized form,
\[
  q_{31}q_{12}s_1
  =
  q_{12}q_{23}s_2
  =
  q_{23}q_{31}s_3
  =
  \frac{D_{123}}{a_1a_2a_3}
  =
  q_{12}+q_{23}+q_{31}-2+\Theta_{123}.
\]
\end{theorem}

\begin{proof}
The identity \(\Sigma_2=a_2D_{123}\) is exactly the expanded numerator
calculation in the Hermitian cross law.  The other two identities are obtained
from the same calculation by cyclic relabelling
\[
  (z_1,z_2,z_3)\mapsto(z_2,z_3,z_1)
  \quad\text{and}\quad
  (z_1,z_2,z_3)\mapsto(z_3,z_1,z_2).
\]
The Gram determinant \(D_{123}\), the sum of quadrances, and
\(\Theta_{123}\) are cyclically invariant by the determinant identity and the
cyclic invariance of the triple product.  Dividing
\(\Sigma_i=a_iD_{123}\) by the appropriate product
\(\Delta_{ji}\Delta_{ik}\) gives the three spread formulas, and dividing
those formulas by \(a_1a_2a_3\) gives the normalized equality.  The final
expression is the Hermitian triangle constraint.
\end{proof}

\begin{corollary}[Cyclic Hermitian Pythagoras]
If the vertex spread at \(z_i\) is \(s_i=1\), then
\[
  q_{ji}q_{ik}
  =
  q_{12}+q_{23}+q_{31}-2+\Theta_{123},
  \qquad
  \{i,j,k\}=\{1,2,3\},
\]
with \(j-i-k\) in the cyclic order used above.  Equivalently,
\[
  \Delta_{ji}\Delta_{ik}=a_iD_{123}.
\]
\end{corollary}

\begin{proof}
This is the cyclic spread law with \(s_i=1\).  The condition \(s_i=1\) is
equivalent to Hermitian orthogonality of the two side polars meeting at
\(z_i\), because \(s_i\) is the projective quadrance of that polar pair.
\end{proof}

\begin{definition}[Ideal boundary triangle]
An ideal triangle is a triple of pairwise transverse null representatives
\(\xi_1,\xi_2,\xi_3\in V_0\), meaning
\[
  h(\xi_i,\xi_i)=0,
  \qquad
  h(\xi_i,\xi_j)\ne0\quad(i\ne j).
\]
Its boundary triple product is
\[
  \mathcal T_\partial(\xi_1,\xi_2,\xi_3)
  =
  -h(\xi_1,\xi_2)h(\xi_2,\xi_3)h(\xi_3,\xi_1).
\]
\end{definition}

\begin{proposition}[Ideal triangle determinant and chain test]
For a pairwise transverse ideal triangle, its Hermitian Gram determinant is
\[
  D_\partial
  =
  h(\xi_1,\xi_2)h(\xi_2,\xi_3)h(\xi_3,\xi_1)
  +
  \overline{h(\xi_1,\xi_2)h(\xi_2,\xi_3)h(\xi_3,\xi_1)}
  =
  -\mathcal T_\partial-\overline{\mathcal T_\partial}.
\]
The phase class of \(\mathcal T_\partial\) is projectively well-defined.
Moreover the three boundary points lie on one chain if and only if
\[
  D_\partial=0,
\]
equivalently \(\mathcal T_\partial\) is purely imaginary.
\end{proposition}

\begin{proof}
The determinant formula is the three-point Gram identity with all diagonal
terms \(h(\xi_i,\xi_i)\) equal to zero.  The scaling law for the triple
product shows that rescaling the \(\xi_i\) multiplies
\(\mathcal T_\partial\) by
\(|\lambda_1\lambda_2\lambda_3|^2\), so its phase class is unchanged.

Since the ambient Hermitian form is nondegenerate on a three-dimensional
space, the Gram determinant of three representatives vanishes exactly when
the representatives are linearly dependent.  A chain is the null boundary of
a complex two-plane.  Thus three boundary points on a chain have linearly
dependent representatives, hence \(D_\partial=0\).  Conversely, if
\(D_\partial=0\), the representatives span a two-dimensional complex plane;
pairwise transversality rules out collapse to one null line and gives the
usual nondegenerate \((1,1)\) chain plane.  Finally
\[
  D_\partial=-\mathcal T_\partial-\overline{\mathcal T_\partial}
\]
vanishes exactly when \(\mathcal T_\partial\) has zero real part.
\end{proof}

\begin{definition}[Pole or chain triangle]
A chain triangle is represented by three non-null positive poles
\(p_1,p_2,p_3\in V_+\).  Its side spreads are
\[
  S_{ij}=q_H([p_i],[p_j]),
\]
and its pole triple phase is the phase class of
\[
  T_+(p_1,p_2,p_3)
  =
  -h(p_1,p_2)h(p_2,p_3)h(p_3,p_1).
\]
\end{definition}

\begin{proposition}[Chain triangle identities]
For a nondegenerate chain triangle, all point-triangle identities above apply
to the positive-pole representatives \(p_1,p_2,p_3\).  In particular, if
\[
  b_i=h(p_i,p_i),
  \qquad
  \Delta^+_{ij}=b_i b_j-|h(p_i,p_j)|^2,
  \qquad
  D_+=\det(h(p_i,p_j)),
\]
and
\[
  \Omega_+=\frac{h(p_1,p_2)h(p_2,p_3)h(p_3,p_1)}{b_1b_2b_3},
  \qquad
  \Theta_+=\Omega_++\overline{\Omega_+},
\]
then
\[
  \frac{D_+}{b_1b_2b_3}
  =
  S_{12}+S_{23}+S_{31}-2+\Theta_+.
\]
The cyclic pole-spread laws are obtained by replacing
\((z_i,a_i,q_{ij},D_{123},\Theta_{123})\) by
\((p_i,b_i,S_{ij},D_+,\Theta_+)\) in the preceding theorem.
\end{proposition}

\begin{proof}
Positive poles are non-null representatives.  Therefore the normalized
triangle constraint, cofactor formulas, cyclic spread law, and Pythagoras
specializations apply verbatim to the triple \(p_1,p_2,p_3\).  The notation
above is exactly the notation of the point-triangle theorem with \(z_i\)
renamed as \(p_i\).
\end{proof}

\begin{proposition}[Chain tangency discriminator]
Let \(p,q\in V_+\) be independent positive poles, and let
\[
  \mathcal C_p=\{[\xi]\in\partial\CH^2:h(\xi,p)=0\},
  \qquad
  \mathcal C_q=\{[\xi]\in\partial\CH^2:h(\xi,q)=0\}.
\]
Then \(\mathcal C_p\) and \(\mathcal C_q\) meet on the boundary if and only if
\[
  \Delta_H(p,q)=0.
\]
In that case the intersection is the null projective line
\[
  \PH\bigl(\Span\{p,q\}^{\perp}\bigr).
\]
\end{proposition}

\begin{proof}
A boundary point lies in \(\mathcal C_p\cap\mathcal C_q\) exactly when it is
represented by a nonzero null vector in \(\Span\{p,q\}^{\perp}\).  Since
\(p,q\) are
independent, this orthogonal complement is one-dimensional.  The restriction
of \(h\) to \(\Span\{p,q\}\) has Gram determinant \(\Delta_H(p,q)\).  If this
determinant is nonzero, the span is nondegenerate, so its orthogonal
complement is a non-null line and contains no boundary point.  If
\(\Delta_H(p,q)=0\), the null-join witness for the pole pair gives the null
radical direction in \(\Span\{p,q\}\).  That radical is contained in
\(\Span\{p,q\}^{\perp}\), and both spaces are one-dimensional, so the
orthogonal complement is exactly the same null projective direction.  Hence
the two chains meet exactly in that boundary point.
\end{proof}

\begin{definition}[Typed mixed triangle readout]
A mixed triangle is a triple whose entries are chosen from non-null points,
boundary null points, and chains represented by positive poles.  Its legal
readouts are assigned pairwise by stratum:
\[
\begin{array}{c|c}
\text{pair type} & \text{legal readout}\\ \hline
\text{non-null/non-null} & q_H,\ \eta_H,\ \Delta_H\\
\text{positive pole/positive pole} & \text{polar spread }q_H\\
\text{null/null} & \text{boundary pair kernel and triple phase}\\
\text{null/non-null} & \text{incidence and pairing kernel }h(\xi,z)\\
\text{null/positive pole} & \text{chain incidence }h(\xi,p)=0
\end{array}
\]
The mixed triple phase, when all three pairings are nonzero, is again the
phase class of
\[
  -h(x_1,x_2)h(x_2,x_3)h(x_3,x_1).
\]
\end{definition}

\begin{proposition}[Mixed triangle projective descent]
Every zero/nonzero incidence statement in the mixed table is independent of
the chosen representatives.  Whenever the mixed triple product is defined,
its phase class is also independent of representatives.
\end{proposition}

\begin{proof}
For a single pairing,
\[
  h(\lambda x,\mu y)=\lambda\overline{\mu}\,h(x,y),
\]
so the property \(h(x,y)=0\) is unchanged by nonzero rescaling of either
representative.  The pair invariants \(q_H,\eta_H,\Delta_H\) on non-null
pairs have already been proved to descend projectively, and polar spread is
the same invariant applied to positive poles.  For the mixed triple product,
the three scalar factors multiply to
\[
  \lambda_1\overline{\lambda_2}\,
  \lambda_2\overline{\lambda_3}\,
  \lambda_3\overline{\lambda_1}
  =
  |\lambda_1\lambda_2\lambda_3|^2,
\]
a positive real norm factor.  Thus the phase class descends.
\end{proof}

\subsection*{Geometric role of the phase}

The normalized phase \(\Omega_{123}\) is not an optional decoration.  The
identity
\[
  \Omega_{123}\overline{\Omega}_{123}
  =
  (1-q_{12})(1-q_{23})(1-q_{31})
\]
says that the three pair quadrances determine the modulus of
\(\Omega_{123}\), but not its argument.  The argument is the Cartan-type
angular readout.  Reversing the cyclic order conjugates \(\Omega_{123}\), so
the imaginary part is the orientation-sensitive part of the triangle.

If \(\Omega_{123}\) is real, the oriented triple product is invariant under
reversal; in the boundary Cartan theory this is the real-circle, or totally
real, extreme.  If \(\Omega_{123}\) is purely imaginary and nonzero, then
\(\Theta_{123}=0\): the phase contributes no real term to the Gram determinant
even though the oriented triangle has maximal \(U(1)\) twist.  On boundary
triples this is the complex-chain extreme.  For interior non-null triples it
should be read algebraically as a phase condition on the Gram data, not by
itself as an incidence statement.

\section{Cross ratio}

The complex cross ratio below is the standard Koranyi--Reimann expression.
The present calculus keeps the full complex projective expression, rather than
only its absolute value, and records the quotient-level algebraic laws it
satisfies.

\begin{definition}[Hermitian cross ratio]
For projective representatives \(a,b,c,d\), define
\[
  X(a,b;c,d)
  =
  \frac{h(c,a)h(d,b)}{h(d,a)h(c,b)},
\]
whenever the denominator is nonzero.
\end{definition}

\begin{proposition}[Scale invariance]
The value \(X(a,b;c,d)\) is unchanged by independently rescaling
\[
  a,b,c,d
  \mapsto
  \lambda_a a,\lambda_b b,\lambda_c c,\lambda_d d
\]
by nonzero complex scalars.
\end{proposition}

\begin{proof}
Each Hermitian pairing acquires one scalar from its first argument and the
conjugate scalar from its second argument:
\[
  h(\lambda_c c,\lambda_a a)=\lambda_c\overline{\lambda_a}h(c,a),
  \qquad
  h(\lambda_d d,\lambda_b b)=\lambda_d\overline{\lambda_b}h(d,b),
\]
while
\[
  h(\lambda_d d,\lambda_a a)=\lambda_d\overline{\lambda_a}h(d,a),
  \qquad
  h(\lambda_c c,\lambda_b b)=\lambda_c\overline{\lambda_b}h(c,b).
\]
Thus the rescaled numerator is
\[
  \lambda_c\lambda_d\overline{\lambda_a}\,\overline{\lambda_b}\,
  h(c,a)h(d,b),
\]
and the rescaled denominator is
\[
  \lambda_d\lambda_c\overline{\lambda_a}\,\overline{\lambda_b}\,
  h(d,a)h(c,b).
\]
The same nonzero scalar multiplies both numerator and denominator, so the
quotient is unchanged.
\end{proof}

\begin{proposition}[Endpoint symmetries]
On the nonzero numerator/denominator domain,
\[
  X(a,b;d,c)=X(a,b;c,d)^{-1},\qquad
  X(b,a;c,d)=X(a,b;c,d)^{-1}.
\]
Swapping both endpoint pairs preserves the value:
\[
  X(b,a;d,c)=X(a,b;c,d).
\]
Exchanging the two ordered pairs conjugates it:
\[
  X(c,d;a,b)=\overline{X(a,b;c,d)}.
\]
\end{proposition}

\begin{proof}
Set
\[
  N=h(c,a)h(d,b),
  \qquad
  D=h(d,a)h(c,b).
\]
On the stated domain \(N\ne0\) and \(D\ne0\), and \(X(a,b;c,d)=N/D\).
Swapping \(c\) and \(d\) interchanges numerator and denominator:
\[
  X(a,b;d,c)
  =
  \frac{h(d,a)h(c,b)}{h(c,a)h(d,b)}
  =
  \frac{D}{N}
  =
  X(a,b;c,d)^{-1}.
\]
Similarly, swapping \(a\) and \(b\) gives
\[
  X(b,a;c,d)
  =
  \frac{h(c,b)h(d,a)}{h(d,b)h(c,a)}
  =
  \frac{D}{N}
  =
  X(a,b;c,d)^{-1}.
\]
Applying both endpoint swaps applies inversion twice.  Equivalently,
\[
  X(b,a;d,c)
  =
  \frac{h(d,b)h(c,a)}{h(c,b)h(d,a)}
  =
  \frac{N}{D}
  =
  X(a,b;c,d).
\]

For the ordered-pair exchange, Hermitian symmetry gives
\[
\begin{aligned}
\overline{X(a,b;c,d)}
&=
\frac{\overline{h(c,a)}\,\overline{h(d,b)}}
     {\overline{h(d,a)}\,\overline{h(c,b)}}\\
&=
\frac{h(a,c)h(b,d)}
     {h(a,d)h(b,c)}.
\end{aligned}
\]
By commutativity of complex multiplication this is
\[
  \frac{h(a,c)h(b,d)}{h(b,c)h(a,d)}
  =
  X(c,d;a,b),
\]
which proves the conjugation law.
\end{proof}

\section{Chain cross-ratio calculus}

Complex geodesics inherit a one-dimensional complex-hyperbolic boundary
calculus.  Algebraically, this is simply the Hermitian cross ratio restricted
to the null lines in one positive-pole hyperplane.

\begin{definition}[Chain cross ratio]
Let \(p\in V_+\), and let
\[
  [\xi_1],[\xi_2],[\xi_3],[\xi_4]\in\mathcal C_p
\]
be boundary points for which the denominator is nonzero.  Define
\[
  X_{\mathcal C_p}(\xi_1,\xi_2;\xi_3,\xi_4)
  =
  X(\xi_1,\xi_2;\xi_3,\xi_4).
\]
\end{definition}

\begin{proposition}[Chain cross-ratio descent and symmetries]
The value \(X_{\mathcal C_p}\) is independent of the representatives
\(\xi_i\) and of the positive-pole representative \(p\).  It satisfies the
same endpoint laws:
\[
  X_{\mathcal C_p}(\xi_1,\xi_2;\xi_4,\xi_3)
  =
  X_{\mathcal C_p}(\xi_1,\xi_2;\xi_3,\xi_4)^{-1},
\]
\[
  X_{\mathcal C_p}(\xi_2,\xi_1;\xi_3,\xi_4)
  =
  X_{\mathcal C_p}(\xi_1,\xi_2;\xi_3,\xi_4)^{-1},
\]
\[
  X_{\mathcal C_p}(\xi_2,\xi_1;\xi_4,\xi_3)
  =
  X_{\mathcal C_p}(\xi_1,\xi_2;\xi_3,\xi_4),
\]
and
\[
  X_{\mathcal C_p}(\xi_3,\xi_4;\xi_1,\xi_2)
  =
  \overline{X_{\mathcal C_p}(\xi_1,\xi_2;\xi_3,\xi_4)}.
\]
\end{proposition}

\begin{proof}
The chain condition \(h(\xi_i,p)=0\) is unchanged by rescaling \(p\), so the
underlying chain \(\mathcal C_p\) depends only on \([p]\).  Once the four
boundary points are on that chain, \(X_{\mathcal C_p}\) is the ambient
Hermitian cross ratio \(X\).  Projective descent in the \(\xi_i\) and all
four endpoint laws are exactly the scale-invariance and endpoint-symmetry
propositions for \(X\).
\end{proof}

\begin{proposition}[Intrinsic chain-plane form]
Let \(W=p^\perp\) with the restricted Hermitian form \(h|_W\).  For
\(\xi_i\in W\), the chain cross ratio computed in the two-dimensional
Hermitian space \(W\) agrees with the ambient value:
\[
  X_{(W,h|_W)}(\xi_1,\xi_2;\xi_3,\xi_4)
  =
  X_{\mathcal C_p}(\xi_1,\xi_2;\xi_3,\xi_4).
\]
\end{proposition}

\begin{proof}
The restricted form \(h|_W\) has the same pairings \(h(\xi_i,\xi_j)\) as the
ambient form for vectors \(\xi_i,\xi_j\in W\).  Substituting these same four
pairings into the cross-ratio formula gives the same quotient.
\end{proof}

\section{Projective-unitary covariance}

The whole calculus is covariant under the projective unitary group.  This is
the symmetry principle that replaces coordinate-dependent normal forms: every
invariant defined only from the Hermitian form is automatically preserved by
maps preserving that form.

\begin{definition}[Hermitian-unitary map]
A Hermitian-unitary map is an invertible complex-linear map \(g:V\to V\)
satisfying
\[
  h(gz,gw)=h(z,w)
  \qquad
  (z,w\in V).
\]
Its projectivization sends \([z]\mapsto[gz]\).
\end{definition}

\begin{proposition}[Projective action and strata]
A Hermitian-unitary map \(g\) descends to a projective automorphism.  It
preserves the three Hermitian strata:
\[
  z\in V_-\Longleftrightarrow gz\in V_-,
  \qquad
  z\in V_0\Longleftrightarrow gz\in V_0,
  \qquad
  z\in V_+\Longleftrightarrow gz\in V_+.
\]
\end{proposition}

\begin{proof}
If \(z'=\lambda z\) with \(\lambda\ne0\), then
\[
  g z'=g(\lambda z)=\lambda gz,
\]
so projectively equivalent representatives remain projectively equivalent.
Thus \(g\) descends to projective classes.  The diagonal value is preserved:
\[
  h(gz,gz)=h(z,z).
\]
Therefore its sign, or its vanishing, is preserved.
\end{proof}

\begin{theorem}[Algebraic projective-unitary covariance]
Let \(g\) be Hermitian-unitary.  On their natural domains,
\[
  \eta_H([gz],[gw])=\eta_H([z],[w]),
  \qquad
  q_H([gz],[gw])=q_H([z],[w]),
\]
\[
  \Delta_H(gz,gw)=\Delta_H(z,w),
  \qquad
  S_H(gp,gq)=S_H(p,q),
\]
\[
  T_H(gz_1,gz_2,gz_3)=T_H(z_1,z_2,z_3),
\]
and
\[
  X(ga,gb;gc,gd)=X(a,b;c,d).
\]
Consequently the following derived objects are projective-unitary invariant:
phase classes, endpoint symmetries, triangle laws, chain triangle laws, and
denominator-cleared polynomial identities built from these quantities.
\end{theorem}

\begin{proof}
Each displayed formula is obtained by substituting
\[
  h(gu,gv)=h(u,v)
\]
into the relevant definition.  For example,
\[
  |h(gz,gw)|^2=|h(z,w)|^2,
  \qquad
  h(gz,gz)h(gw,gw)=h(z,z)h(w,w),
\]
so \(\eta_H\), \(q_H\), and \(\Delta_H\) are unchanged.  The polar spread
formula is the same pair formula applied to positive poles.  The triple
product satisfies
\[
\begin{aligned}
T_H(gz_1,gz_2,gz_3)
&=-h(gz_1,gz_2)h(gz_2,gz_3)h(gz_3,gz_1)\\
&=-h(z_1,z_2)h(z_2,z_3)h(z_3,z_1),
\end{aligned}
\]
and the cross-ratio formula is unchanged because all four pairings in the
quotient are unchanged.  The final statement follows because the listed laws
are equations formed from these invariant quantities.
\end{proof}

\begin{proposition}[Incidence and chain covariance]
For nonzero \(x,p\in V\),
\[
  h(x,p)=0
  \quad\Longleftrightarrow\quad
  h(gx,gp)=0.
\]
Hence
\[
  g(L_p)=L_{gp},
  \qquad
  g(\mathcal C_p)=\mathcal C_{gp},
\]
and chain cross-ratios are invariant:
\[
  X_{\mathcal C_{gp}}(g\xi_1,g\xi_2;g\xi_3,g\xi_4)
  =
  X_{\mathcal C_p}(\xi_1,\xi_2;\xi_3,\xi_4).
\]
\end{proposition}

\begin{proof}
The incidence equivalence is exactly
\[
  h(gx,gp)=h(x,p).
\]
Together with stratum preservation, this maps negative lines in \(p^\perp\)
to negative lines in \((gp)^\perp\), and null lines in \(p^\perp\) to null
lines in \((gp)^\perp\).  Chain cross-ratio invariance is the ambient
cross-ratio covariance applied to four null representatives lying on
\(\mathcal C_p\).
\end{proof}

\begin{proposition}[Boundary contact covariance]
The projectivized action of a Hermitian-unitary map preserves the CR/contact
distribution on \(\partial\CH^2\).  On the null cone, the tautological contact
form
\[
  \alpha_\xi(v)=\operatorname{Im} h(v,\xi)
\]
is pulled back to itself:
\[
  \alpha_{g\xi}(gv)=\alpha_\xi(v).
\]
Thus the induced boundary map preserves the contact class and carries chains
to chains.
\end{proposition}

\begin{proof}
Since \(g\) preserves \(h\),
\[
  \alpha_{g\xi}(gv)
  =
  \operatorname{Im} h(gv,g\xi)
  =
  \operatorname{Im} h(v,\xi)
  =
  \alpha_\xi(v).
\]
The null cone is preserved by the stratum proposition, so this identity
descends to the projectivized boundary contact distribution.  Chains are the
projectivized null parts of positive-pole hyperplanes, and these were shown
above to map as \(\mathcal C_p\mapsto\mathcal C_{gp}\).
\end{proof}

\begin{proposition}[Ball-kernel automorphy covariance]
In a ball chart, let a projective-unitary transformation have automorphy
factor \(J_g(z)\ne0\) satisfying
\[
  B(gz,gw)
  =
  \frac{B(z,w)}{J_g(z)\overline{J_g(w)}}.
\]
For any exponent \(\gamma\) for which the powers are defined multiplicatively,
\[
  K_\gamma(gz,gw)
  =
  \bigl(J_g(z)\overline{J_g(w)}\bigr)^\gamma K_\gamma(z,w).
\]
Moreover the left-normalized overlap satisfies
\[
  \Omega_\gamma(gz,gw)
  =
  \left(\frac{\overline{J_g(w)}}{\overline{J_g(z)}}\right)^\gamma
  \Omega_\gamma(z,w).
\]
\end{proposition}

\begin{proof}
The first identity follows directly from
\[
  K_\gamma(z,w)=B(z,w)^{-\gamma}.
\]
Indeed,
\[
\begin{aligned}
K_\gamma(gz,gw)
&=
\left(\frac{B(z,w)}{J_g(z)\overline{J_g(w)}}\right)^{-\gamma}\\
&=
\bigl(J_g(z)\overline{J_g(w)}\bigr)^\gamma K_\gamma(z,w).
\end{aligned}
\]
Setting \(w=z\) gives
\[
  K_\gamma(gz,gz)
  =
  |J_g(z)|^{2\gamma}K_\gamma(z,z).
\]
Dividing the first covariance formula by this diagonal formula gives the
left-normalized overlap law.
\end{proof}

\section{Bisectors and spinal spheres}

Bisectors are the first place where the calculus must deliberately avoid a
metric quotient.  The interior bisector is an equal-kernel locus.  Its boundary
trace is still meaningful after the normalizing factor \(h(x,x)\) has vanished,
provided the equation has first been cleared of denominators.

\begin{definition}[Hermitian bisector polynomial]
Let \(a,b\in V\) be non-null focus vectors.  Define
\[
  \mathcal B_{a,b}(x)
  =
  h(x,a)h(a,x)h(b,b)-h(x,b)h(b,x)h(a,a).
\]
The algebraic bisector is the projective zero locus
\[
  \mathfrak B(a,b)
  =
  \{[x]\in\PH(V\setminus\{0\}):\mathcal B_{a,b}(x)=0\}.
\]
Its complex-hyperbolic bisector is
\[
  \operatorname{Bis}(a,b)=\mathfrak B(a,b)\cap\CH^2,
\]
and its spinal sphere is the boundary trace
\[
  \mathfrak S(a,b)=\mathfrak B(a,b)\cap\partial\CH^2.
\]
\end{definition}

\begin{proposition}[Projective descent and focus symmetry]
For nonzero scalars \(\lambda,\mu,\nu\),
\[
  \mathcal B_{\lambda a,\mu b}(\nu x)
  =
  |\lambda\mu\nu|^2\,\mathcal B_{a,b}(x).
\]
Hence \(\mathfrak B(a,b)\), \(\operatorname{Bis}(a,b)\), and
\(\mathfrak S(a,b)\) depend only on the projective focus classes \([a]\) and
\([b]\).  Also
\[
  \mathcal B_{b,a}(x)=-\mathcal B_{a,b}(x),
\]
so the zero locus is unchanged by swapping the two foci.
\end{proposition}

\begin{proof}
Using Hermitian sesquilinearity,
\[
  h(\nu x,\lambda a)h(\lambda a,\nu x)
  =
  |\nu\lambda|^2 h(x,a)h(a,x),
\]
and
\[
  h(\mu b,\mu b)=|\mu|^2h(b,b).
\]
Thus the first term in \(\mathcal B_{\lambda a,\mu b}(\nu x)\) is multiplied
by \(|\lambda\mu\nu|^2\).  The same computation with \(a\) and \(b\)
interchanged gives the same factor for the second term.  Since the factor is
nonzero, the equation \(\mathcal B_{a,b}(x)=0\) descends to projective
classes.  Swapping \(a\) and \(b\) simply reverses the two terms in the
difference.
\end{proof}

\begin{proposition}[Kernel-equidistance form]
For non-null foci \(a,b\), the bisector equation is equivalent to
\[
  \frac{h(x,a)h(a,x)}{h(a,a)}
  =
  \frac{h(x,b)h(b,x)}{h(b,b)}.
\]
If \(x\) is also non-null, this is equivalently
\[
  \eta_H([x],[a])=\eta_H([x],[b]).
\]
\end{proposition}

\begin{proof}
Because \(a\) and \(b\) are non-null, \(h(a,a)\) and \(h(b,b)\) are nonzero.
Multiplying the displayed equal-kernel equation by \(h(a,a)h(b,b)\) gives
\(\mathcal B_{a,b}(x)=0\).  If \(x\) is non-null, then
\[
  \eta_H([x],[a])
  =
  \frac{h(x,a)h(a,x)}{h(x,x)h(a,a)},
  \qquad
  \eta_H([x],[b])
  =
  \frac{h(x,b)h(b,x)}{h(x,x)h(b,b)}.
\]
The common nonzero factor \(h(x,x)\) cancels.
\end{proof}

\begin{theorem}[Projective-unitary covariance of bisectors]
Let \(g\) be Hermitian-unitary.  Then
\[
  \mathcal B_{ga,gb}(gx)=\mathcal B_{a,b}(x).
\]
Consequently
\[
  g(\mathfrak B(a,b))=\mathfrak B(ga,gb),\qquad
  g(\operatorname{Bis}(a,b))=\operatorname{Bis}(ga,gb),
\]
and
\[
  g(\mathfrak S(a,b))=\mathfrak S(ga,gb).
\]
\end{theorem}

\begin{proof}
Substitute \(h(gu,gv)=h(u,v)\) into the definition of
\(\mathcal B\):
\[
\begin{aligned}
\mathcal B_{ga,gb}(gx)
&=
h(gx,ga)h(ga,gx)h(gb,gb)\\
&\qquad
-h(gx,gb)h(gb,gx)h(ga,ga)\\
&=
h(x,a)h(a,x)h(b,b)-h(x,b)h(b,x)h(a,a).
\end{aligned}
\]
This is \(\mathcal B_{a,b}(x)\).  The projective zero loci follow from this
identity and from the fact that \(g\) preserves the negative and null strata.
\end{proof}

\begin{proposition}[Spinal sphere in the finite Heisenberg chart]
Use the Siegel Hermitian form
\[
  H(Z,W)=Z_0\overline{W_2}+Z_1\overline{W_1}+Z_2\overline{W_0}
\]
and the finite null lift
\[
  \widehat\xi(\zeta,t)
  =
  \left(\frac{-|\zeta|^2+i t}{2},\,\zeta,\,1\right).
\]
For a focus \(f=(f_0,f_1,f_2)\), set
\[
  A_f(\zeta,t)
  =
  H(\widehat\xi(\zeta,t),f)
  =
  \frac{-|\zeta|^2+i t}{2}\overline{f_2}
  +\zeta\overline{f_1}
  +\overline{f_0},
  \qquad
  n_f=H(f,f).
\]
For non-null foci \(a,b\), the finite part of the spinal sphere is
\[
  \Phi_{a,b}(\zeta,t)
  =
  n_b A_a(\zeta,t)\overline{A_a(\zeta,t)}
  -
  n_a A_b(\zeta,t)\overline{A_b(\zeta,t)}
  =
  0.
\]
The point at infinity lies on the spinal sphere exactly when
\[
  n_b|a_2|^2-n_a|b_2|^2=0.
\]
\end{proposition}

\begin{proof}
The spinal sphere is the null-boundary zero locus of
\(\mathcal B_{a,b}\).  Substituting \(x=\widehat\xi(\zeta,t)\) into the
bisector polynomial for the Siegel form gives
\[
  H(\widehat\xi,a)H(a,\widehat\xi)H(b,b)
  -
  H(\widehat\xi,b)H(b,\widehat\xi)H(a,a)=0.
\]
By Hermitian symmetry,
\[
  H(\widehat\xi,f)=A_f(\zeta,t),
  \qquad
  H(f,\widehat\xi)=\overline{A_f(\zeta,t)},
  \qquad
  H(f,f)=n_f.
\]
This is exactly \(\Phi_{a,b}(\zeta,t)=0\).

For the point at infinity, use the null lift
\(\widehat\xi(\infty)=(1,0,0)\).  Then
\[
  H(\widehat\xi(\infty),f)=\overline{f_2}.
\]
Substitution in the same denominator-cleared equation gives
\[
  n_b\,\overline{a_2}a_2-n_a\,\overline{b_2}b_2=0,
\]
which is the stated condition.
\end{proof}

\section{Boundary contact calculus}

The boundary \(\partial\CH^2=\PH(V_0)\) carries a CR/contact geometry.  The
boundary primitive is not a real angle.  It is the package of null projective
incidence, boundary cross ratios, triple-product phase, and contact-volume
normalization intrinsic to the CR sphere.  This is a readout package on the
standard CR boundary, not a replacement for the CR geometry itself.

\begin{definition}[Boundary contact readout]
A boundary contact readout is the package:
\[
\begin{array}{ll}
\text{null incidence} & h(\xi,\eta)=0\text{ or }h(\xi,\eta)\ne0,\\
\text{boundary cross-ratio data} & X(\xi_1,\xi_2;\xi_3,\xi_4),\\
\text{triple phase class} & \argg[-h(\xi_1,\xi_2)h(\xi_2,\xi_3)h(\xi_3,\xi_1)],\\
\text{contact normalization} & \text{the standard CR/contact volume class.}
\end{array}
\]
\end{definition}

\begin{definition}[Boundary defining contrast]
In the ball chart \(B^2=\{z\in\C^2:\|z\|^2<1\}\), set
\[
  \delta(z)=1-\|z\|^2,
  \qquad
  \Theta(z)=-\frac32\log\delta(z).
\]
The function \(\delta\) is a boundary defining function, and \(\Theta\) is its
logarithmic contrast.
\end{definition}

\begin{proposition}[Boundary-tangent contrast criterion]
Let \(X\) be a smooth real vector field on a neighborhood in the ball chart.
On the interior of \(B^2\),
\[
  X\Theta=0
  \quad\Longleftrightarrow\quad
  X\delta=0.
\]
Consequently \(X\Theta=0\) is precisely the condition that \(X\) be tangent
to the level hypersurfaces of \(\delta\).  Along any locus on which
\(X\Theta=0\), the additional equations
\[
  X(X\Theta)=0,
  \qquad
  Y(X\Theta)=0
\]
say respectively that the \(X\)-flow has no second-order contrast component
and that the contrast condition is preserved to first order in the direction
\(Y\).
\end{proposition}

\begin{proof}
The first claim is the differential identity
\[
  X\Theta
  =
  -\frac32\,\frac{X\delta}{\delta}.
\]
Since \(\delta>0\) in the ball, this vanishes exactly when \(X\delta=0\).
The level hypersurfaces of \(\delta\) have tangent distribution
\(\ker\dd\delta\), so this is the ordinary tangency criterion.

If \(X\Theta=0\) on a locus, then differentiating \(X\Theta\) in the \(X\)
direction gives \(X(X\Theta)=0\) exactly when the vanishing persists through
second order along the \(X\)-flow.  Differentiating the same scalar function
in the direction \(Y\) gives the stated first-order preservation condition.
\end{proof}

\section{Heisenberg boundary chart}

The Heisenberg compactification and its standard transformations are
background.  The calculus uses this chart to write explicit boundary pair
kernels and covariance laws.

For the analytic \(\CH^2\) boundary chart, use the Siegel Hermitian form
\[
  H(Z,W)=Z_0\overline{W_2}+Z_1\overline{W_1}+Z_2\overline{W_0}.
\]
The one-point compactified Heisenberg boundary is
\[
  (\C\times\mathbb R)\cup\{\infty\}.
\]
The finite point \((\zeta,t)\) has homogeneous lift
\[
  \widehat\xi(\zeta,t)
  =
  \left(\frac{-|\zeta|^2+i t}{2},\,\zeta,\,1\right),
\]
and the point at infinity has lift
\[
  \widehat\xi(\infty)=(1,0,0).
\]
These lifts are null.  For finite points \(p=(\zeta,t)\) and \(q=(\eta,s)\),
\[
\begin{aligned}
  H(\widehat\xi(p),\widehat\xi(q))
  &=
  \zeta\overline{\eta}
  -\frac{|\zeta|^2+|\eta|^2}{2}
  +\frac{i(t-s)}{2} \\
  &=
  -\frac12\left(
    |\zeta-\eta|^2
    -i(t-s+2\,\operatorname{Im}(\zeta\overline\eta))
  \right).
\end{aligned}
\]
Thus the boundary pair kernel contains both horizontal separation and contact
phase.
It is also naturally covariant under the affine Heisenberg similarities.  If
\(L_g\) denotes left translation, \(R_{c,s}\) a unit horizontal rotation
(\(c^2+s^2=1\)), and \(\delta_r(\zeta,t)=(r\zeta,r^2t)\), then the
finite-chart identities are
\[
  K(L_gp,L_gq)=K(p,q),\qquad
  K(R_{c,s}p,R_{c,s}q)=K(p,q),
\]
and
\[
  K(\delta_r p,\delta_r q)=r^2K(p,q).
\]
Consequently the finite similarity \(p\mapsto g\cdot R_{c,s}(\delta_rp)\)
scales the boundary pair kernel by \(r^2\).  This gives the boundary-kernel
part of the calculus in intrinsic CR/contact and kernel terms.

The standard contact form on the finite chart is
\[
  \theta=\dd t+2\,\operatorname{Im}(\overline\zeta\,\dd\zeta).
\]
If \(\zeta=x+i y\), then
\[
  \theta=\dd t+2(x\,\dd y-y\,\dd x),
  \qquad
  \dd\theta=4\,\dd x\wedge\dd y,
\]
and, for the orientation \(\dd t\wedge\dd x\wedge\dd y\),
\[
  \theta\wedge\dd\theta
  =
  4\,\dd t\wedge\dd x\wedge\dd y.
\]
This is the finite-chart density for the standard contact-volume class.  A
compactified boundary calculation treats \(\infty\) by passing to a second
Heisenberg chart, as in the atlas transition formulas below.

\begin{proposition}[Anti-complex conjugation character]
Let \(C_m:\C^m\to\C^m\) be coordinatewise conjugation,
\[
  C_m(z_1,\ldots,z_m)=(\overline z_1,\ldots,\overline z_m).
\]
As a real linear map on \(\mathbb R^{2m}\),
\[
  \det_{\mathbb R}(C_m)=(-1)^m.
\]
In the finite Heisenberg chart, define
\[
  \kappa(\zeta,t)=(\overline\zeta,-t),
  \qquad
  \text{equivalently}\qquad
  \kappa(x,y,t)=(x,-y,-t).
\]
Then
\[
  \widehat\xi(\kappa(\zeta,t))
  =
  \overline{\widehat\xi(\zeta,t)}
\]
and
\[
  H(\widehat\xi(\kappa p),\widehat\xi(\kappa q))
  =
  \overline{H(\widehat\xi(p),\widehat\xi(q))}.
\]
For the standard contact form,
\[
  \kappa^\ast\theta=-\theta,
  \qquad
  \kappa^\ast(\dd\theta)=-\dd\theta,
  \qquad
  \kappa^\ast(\theta\wedge\dd\theta)=\theta\wedge\dd\theta.
\]
\end{proposition}

\begin{proof}
For one complex coordinate \(z=x+i y\), conjugation is the real map
\[
  (x,y)\longmapsto(x,-y),
\]
whose determinant is \(-1\).  Taking the product over \(m\) independent
coordinates gives
\[
  \det_{\mathbb R}(C_m)=(-1)^m.
\]

The lift identity follows directly:
\[
\begin{aligned}
\widehat\xi(\overline\zeta,-t)
&=
\left(\frac{-|\zeta|^2-i t}{2},\,\overline\zeta,\,1\right)\\
&=
\overline{
  \left(\frac{-|\zeta|^2+i t}{2},\,\zeta,\,1\right)
}.
\end{aligned}
\]
Since the Siegel Hermitian form has real coefficients,
\[
  H(\overline Z,\overline W)=\overline{H(Z,W)}.
\]
Substituting \(Z=\widehat\xi(p)\) and \(W=\widehat\xi(q)\) gives the displayed
kernel-conjugation formula.

Finally, in real coordinates,
\[
  \theta=\dd t+2(x\,\dd y-y\,\dd x).
\]
Under \(\kappa(x,y,t)=(x,-y,-t)\),
\[
  \kappa^\ast\dd t=-\dd t,\qquad
  \kappa^\ast\dd x=\dd x,\qquad
  \kappa^\ast\dd y=-\dd y.
\]
Therefore
\[
\begin{aligned}
\kappa^\ast\theta
&=
-\dd t+2\bigl(x(-\dd y)-(-y)\dd x\bigr)\\
&=
-\dd t-2(x\,\dd y-y\,\dd x)\\
&=
-\theta.
\end{aligned}
\]
Taking exterior derivatives gives
\[
  \kappa^\ast(\dd\theta)=\dd(\kappa^\ast\theta)=-\dd\theta,
\]
and hence
\[
  \kappa^\ast(\theta\wedge\dd\theta)
  =
  (-\theta)\wedge(-\dd\theta)
  =
  \theta\wedge\dd\theta.
\]
\end{proof}

\begin{figure}[ht]
\centering
\begin{tikzpicture}[
    x=1cm,y=1cm,
    line cap=round,
    line join=round,
    >=Latex,
    label/.style={font=\small, fill=white, inner sep=1.5pt, rounded corners=2pt},
    note/.style={font=\small}
  ]
  \draw[->, thick, gray!80!black] (-3.50,0) -- (3.50,0)
    node[right, text=black] {\(x\)};
  \draw[->, thick, gray!80!black] (0,-0.50) -- (0,3.50)
    node[above, text=black] {\(t\)};
  \draw[->, thick, gray!80!black] (-2.20,-1.20) -- (2.20,1.20)
    node[right, text=black] {\(y\)};

  \fill[blue!10, opacity=0.8] (-2.10,0.55) ellipse (1.08 and 0.38);
  \draw[blue!65, thick] (-2.10,0.55) ellipse (1.08 and 0.38);
  \node[label, text=blue!70!black] at (-2.65,1.12) {\(\theta=0\)};
  \draw[dashed, gray!70] (-2.10,0.55) -- (-2.10,0);

  \fill[blue!10, opacity=0.8, rotate around={18:(0.80,1.30)}]
    (0.80,1.30) ellipse (0.94 and 0.30);
  \draw[blue!65, thick, rotate around={18:(0.80,1.30)}]
    (0.80,1.30) ellipse (0.94 and 0.30);
  \node[label, text=blue!70!black] at (2.40,1.85) {contact plane};
  \draw[blue!65, ->, shorten >=2pt] (1.85,1.65) -- (1.25,1.45);
  \draw[dashed, gray!70] (0.80,1.30) -- (0.80,0);

  \shade[inner color=orange!15, outer color=orange!50, opacity=0.85]
    (0.20,1.30) ellipse (1.18 and 0.88);
  \draw[orange!80!black, thick] (0.20,1.30) ellipse (1.18 and 0.88);
  \node[label, text=orange!90!black] at (0.20,2.65)
    {spinal sphere \(\mathfrak S\)};

  \draw[very thick,red!70!black]
    (-2.35,0.16) .. controls (-1.50,2.42) and (1.38,2.50) .. (2.35,0.38);
  \node[label, text=red!70!black] at (2.82,0.80) {chain};
  \draw[red!70!black, ->, shorten >=2pt] (2.46,0.72) -- (2.08,0.70);

  \coordinate (Qone) at (-0.62,1.90);
  \coordinate (Qtwo) at (1.20,0.72);
  \fill[black] (Qone) circle (2pt);
  \fill[black] (Qtwo) circle (2pt);
  \node[label, above right=1pt] at (Qone) {\(q_1\)};
  \node[label, below right=1pt] at (Qtwo) {\(q_2\)};
  \node[label] at (-2.35,2.48) {\(\mathcal C_p\cap\mathfrak S=\{q_1,q_2\}\)};
  \draw[gray!60,->] (-1.28,2.34) -- (Qone);
  \draw[gray!60,->] (-1.28,2.34) .. controls (-0.18,2.18) and (0.78,1.46) .. (Qtwo);
\end{tikzpicture}
\caption{Schematic Heisenberg boundary chart.  The boundary is represented by
coordinates \((\zeta,t)=(x+iy,t)\).  Contact planes are the horizontal
distribution \(\theta=0\); chains appear as CR circles; and spinal spheres
are boundary traces of bisectors.  Their intersections are governed by the
same Hermitian incidence and kernel equations as in the projective model.}
\label{fig:heisenberg-spinal}
\end{figure}
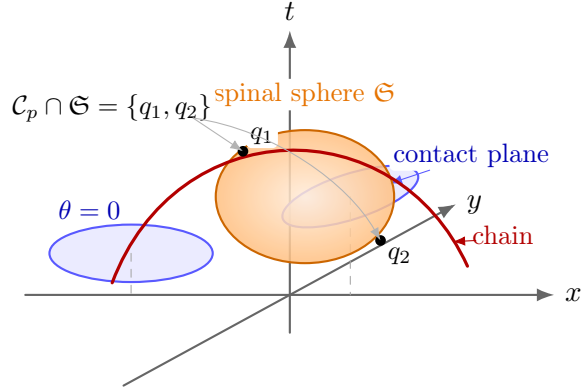

\begin{proposition}[Finite Heisenberg contact density]
In finite Heisenberg coordinates \((t,x,y)\), with
\[
  \theta=\dd t+2(x\,\dd y-y\,\dd x),
\]
one has
\[
  \dd\theta=4\,\dd x\wedge\dd y,\qquad
  \theta\wedge\dd\theta
  =
  4\,\dd t\wedge\dd x\wedge\dd y.
\]
\end{proposition}

\begin{proof}
Since \(\dd(\dd t)=0\),
\[
  \dd\theta
  =
  2(\dd x\wedge\dd y-\dd y\wedge\dd x)
  =
  4\,\dd x\wedge\dd y.
\]
The terms involving \(\dd x\wedge\dd x\) and \(\dd y\wedge\dd y\) vanish, so
\[
  \theta\wedge\dd\theta
  =
  \dd t\wedge 4\,\dd x\wedge\dd y.
\]
\end{proof}

\begin{proposition}[Finite Heisenberg contact invariance]
Let \(g=(a+i b,c)\) and \(p=(x+i y,t)\).  With finite Heisenberg left
translation
\[
  L_g(p)=g\cdot p
  =
  \bigl(a+x+i(b+y),\,c+t+2(bx-a y)\bigr),
\]
one has
\[
  L_g^\ast\theta=\theta.
\]
Consequently the finite contact-volume density
\(\theta\wedge\dd\theta\) is invariant under finite Heisenberg left
translation.
\end{proposition}

\begin{proof}
Write \(x'=a+x\), \(y'=b+y\), and
\[
  t'=c+t+2(bx-a y).
\]
Then
\[
  \dd t'=\dd t+2(b\,\dd x-a\,\dd y),
  \qquad
  \dd x'=\dd x,
  \qquad
  \dd y'=\dd y.
\]
Therefore
\[
\begin{aligned}
L_g^\ast\theta
&=
\dd t'+2(x'\,\dd y'-y'\,\dd x')\\
&=
\dd t+2(b\,\dd x-a\,\dd y)
  +2((a+x)\dd y-(b+y)\dd x)\\
&=
\dd t+2(x\,\dd y-y\,\dd x)
=\theta.
\end{aligned}
\]
Taking \(\theta\wedge\dd\theta\) gives the density statement.
\end{proof}

\begin{proposition}[Horizontal rotation contact invariance]
Let \(c,s\in\mathbb R\) satisfy \(c^2+s^2=1\), and let
\[
  R_{c,s}(x,y,t)=(c x-s y,\,s x+c y,\,t).
\]
Then
\[
  R_{c,s}^\ast\theta=\theta,
\]
and hence \(\theta\wedge\dd\theta\) is invariant under unit horizontal
rotations.
\end{proposition}

\begin{proof}
Write \(x'=c x-s y\), \(y'=s x+c y\), and \(t'=t\).  Then
\[
  \dd x'=c\,\dd x-s\,\dd y,\qquad
  \dd y'=s\,\dd x+c\,\dd y.
\]
Thus
\[
\begin{aligned}
R_{c,s}^\ast\theta
&=
\dd t+2(x'\dd y'-y'\dd x')\\
&=
\dd t+2(c^2+s^2)(x\,\dd y-y\,\dd x)\\
&=\theta.
\end{aligned}
\]
\end{proof}

\begin{proposition}[Heisenberg dilation contact scaling]
For \(r\in\mathbb R\), let
\[
  \delta_r(x,y,t)=(r x,r y,r^2 t).
\]
Then
\[
  \delta_r^\ast\theta=r^2\theta,
  \qquad
  \delta_r^\ast(\dd\theta)=r^2\dd\theta.
\]
Consequently
\[
  \delta_r^\ast(\theta\wedge\dd\theta)
  =
  r^4\,\theta\wedge\dd\theta.
\]
\end{proposition}

\begin{proof}
Writing \(x'=r x\), \(y'=r y\), and \(t'=r^2t\), we have
\[
  \dd x'=r\,\dd x,\qquad
  \dd y'=r\,\dd y,\qquad
  \dd t'=r^2\,\dd t.
\]
Therefore
\[
\begin{aligned}
\delta_r^\ast\theta
&=
r^2\dd t+2(r x)(r\,\dd y)-2(r y)(r\,\dd x)\\
&=
r^2\bigl(\dd t+2(x\,\dd y-y\,\dd x)\bigr)
=r^2\theta.
\end{aligned}
\]
Taking exterior derivatives gives the \(r^2\) scaling of \(\dd\theta\), and
wedging the two scaled forms gives the \(r^4\) density scaling.
\end{proof}

These finite-chart laws are naturally expressed as constant-weight
contact-conformal transformations.  If a finite Heisenberg map \(\phi\)
satisfies
\[
  \phi^\ast\theta=w\theta,\qquad
  \phi^\ast(\dd\theta)=w\,\dd\theta,
\]
then its top-degree rule is
\[
  (\phi^\ast\theta)\wedge(\phi^\ast\dd\theta)
  =
  w^2\,\theta\wedge\dd\theta.
\]
Translations and unit horizontal rotations have \(w=1\), while the dilation
\(\delta_r\) has \(w=r^2\).

The affine transformations form the standard Heisenberg similarity calculus.
Left translations satisfy the Heisenberg identity, inverse, and associativity
laws; horizontal rotations compose by the usual angle-addition formulas and
preserve \(c^2+s^2=1\); and dilations compose by multiplying their scales.
Rotations and dilations conjugate translations by rotating or dilating the
translation parameter, and rotations commute with dilations in the finite
chart.  Equivalently, a translation, a unit rotation, and a dilation combine
into a finite Heisenberg similarity
\[
  p\longmapsto g\cdot R_{c,s}(\delta_r p),
\]
whose contact weight is \(r^2\) and whose contact-volume coefficient scales by
\(r^4\).  The similarity maps are closed under composition: the second
translation parameter is first dilated and rotated by the linear part of the
first similarity, the two translation parameters multiply in the Heisenberg
group, the rotation parameters compose by angle addition, and the dilation
scales multiply.  On finite boundary invariants, the same similarity rescales
the pair kernel by \(r^2\).  Hence, when \(r\ne0\), finite cross-ratios are
preserved; finite triple products are multiplied by \((r^2)^3\), so their
phase class is unchanged.  The same calculation covers the first compactified
reductions involving \(\infty\): because these similarities fix \(\infty\),
the first-infinity cross-ratio quotient is preserved and the first-infinity
triple product is multiplied by \(r^2\).

The same affine maps extend to the one-point compactification by fixing
\(\infty\).  Thus finite points remain in the finite chart and the finite
pullback coefficient is governed by the same contact weight.  For compactified
similarities the retained finite-chart contact weight is \(r^2\), so the top
coefficient still scales by \(r^4\).  The compactified affine maps are also
closed under composition:
their boundary maps compose, finite points remain finite, \(\infty\) remains
fixed, and their contact weights multiply.  For compactified similarities, the
boundary composition law is the same as the finite-chart similarity law: the
new translation parameter is
\[
  g_1\cdot R_{c_1,s_1}(\delta_{r_1}g_2),
\]
the new rotation parameters are
\[
  (c_1c_2-s_1s_2,\ s_1c_2+c_1s_2),
\]
and the new dilation scale is \(r_1r_2\).  The compactified-composite weight
is therefore \((r_1r_2)^2\), agreeing with the weight of the composed
similarity.  Finite and first-infinity cross-ratios use the same preservation
laws as above, while finite and first-infinity triple products carry the
corresponding \(r^2\)-weight powers.  This is the affine part of the
compactified boundary atlas.

For inversion and other second-chart overlaps, the transition is not a global
affine map with constant weight.  The overlap is partial in a finite chart,
because one domain pole is omitted, and its contact weight may depend on the
point.  The appropriate law is therefore a partial contact-conformal chart
transition:
\[
  \phi^\ast\theta=\lambda(p)\theta
\]
on the finite overlap.  Since
\(\dd(\lambda\theta)=\dd\lambda\wedge\theta+\lambda\dd\theta\), the
top-degree contact-volume term still satisfies
\[
  \phi^\ast(\theta\wedge\dd\theta)
  =
  \lambda(p)^2\,\theta\wedge\dd\theta.
\]
The compactified transition sends the omitted domain pole to the omitted
image pole.  The standard inversion is the case in which the omitted finite
pole maps to \(\infty\).  On the punctured finite chart, for
\(\zeta=x+iy\), \(r=x^2+y^2\), and \(D=r^2+t^2\), it is
\[
  x'=\frac{2(-rx+ty)}{D},\qquad
  y'=\frac{2(-tx-ry)}{D},\qquad
  t'=\frac{-4t}{D},
\]
with pointwise contact weight
\[
  \lambda(p)=\frac{4}{D}.
\]
The compactified map sends the finite origin to \(\infty\) and sends
\(\infty\) to the finite origin.  These coordinates are exactly the Siegel
projective coordinate swap
\((Z_0,Z_1,Z_2)\mapsto(Z_2,Z_1,Z_0)\) on punctured finite null lifts.  The
finite pole and involution laws are
\[
  D(p)=0\Longleftrightarrow p=(0,0),\qquad
  D(I(p))=\frac{16}{D(p)},\qquad
  I(I(p))=p
\]
on the punctured finite chart.  Consequently the compactified boundary
inversion satisfies \(I_{\partial}^2=\mathrm{id}\).  Since the Siegel form is
linear in its first slot and conjugate-linear in its second slot, the same
scalar identity gives
\[
  K(I(p),I(q))
  =
  \alpha(p)\overline{\alpha(q)}\,K(p,q)
\]
for finite boundary pair kernels.  The four scalar factors cancel in the
standard cross-ratio, so \(X(I(a),I(b);I(c),I(d))=X(a,b;c,d)\) whenever the
four finite points avoid the inversion pole.  For triples, the same
calculation gives
\[
  T(I(a),I(b),I(c))
  =
  \alpha(a)\overline{\alpha(b)}
  \alpha(b)\overline{\alpha(c)}
  \alpha(c)\overline{\alpha(a)}\,T(a,b,c).
\]
The multiplier is
\[
  |\alpha(a)|^2|\alpha(b)|^2|\alpha(c)|^2,
\]
so the Cartan phase class is preserved modulo norm scalars.  The compactified
boundary-map calculus also includes the first-infinity reductions.  If
\(I_{\partial}(\infty)=0\), then
\[
  X(I_{\partial}\infty,I(b);I(c),I(d))
  =
  X(\infty,b;c,d),
\]
and
\[
  T(I_{\partial}\infty,I(p),I(q))
  =
  |\alpha(p)|^2|\alpha(q)|^2\,T(\infty,p,q),
\]
again on the punctured finite entries.  Thus the inversion atlas preserves
cross-ratio formulas and preserves Cartan phase classes after the expected
norm-scalar rescaling.

\begin{proposition}[Standard inversion as Siegel coordinate swap]
Let \(p=(x+i y,t)\ne(0,0)\), set \(r=x^2+y^2\) and
\(D=r^2+t^2\), and define \(I(p)=(x'+iy',t')\) by
\[
  x'=\frac{2(-rx+ty)}{D},\qquad
  y'=\frac{2(-tx-ry)}{D},\qquad
  t'=\frac{-4t}{D}.
\]
Let
\[
  \sigma(Z_0,Z_1,Z_2)=(Z_2,Z_1,Z_0).
\]
Then \(\sigma\) preserves the Siegel Hermitian form \(H\), exchanges
\(\widehat\xi(\infty)\) and \(\widehat\xi(0,0)\), and on the punctured finite
chart
\[
  \widehat\xi(I(p))
  =
  \alpha(p)\,\sigma(\widehat\xi(p)),
  \qquad
  \alpha(p)=\frac{-2r-2it}{D}.
\]
Consequently the explicit Heisenberg inversion is the boundary projective map
induced by the linear isometry \(\sigma\).  Moreover \(D=0\) only at the
finite origin, \(D(I(p))=16/D\), and \(I(I(p))=p\); hence the induced
compactified boundary map is an involution.  Finally, sesquilinearity gives
\[
  K(I(p),I(q))
  =
  \alpha(p)\overline{\alpha(q)}K(p,q),
\]
and substituting this into the four pairings defining \(X(a,b;c,d)\) cancels
all four scalar factors.  Substituting it into the three pairings defining the
Cartan triple product gives the displayed triple-product multiplier; after
commuting factors, that multiplier is the product of the three norm scalars
\(\alpha(a)\overline{\alpha(a)}\),
\(\alpha(b)\overline{\alpha(b)}\), and
\(\alpha(c)\overline{\alpha(c)}\).  Pairing an inverted finite point with the
finite origin gives \(\alpha(p)\), and pairing the origin with it gives
\(\overline{\alpha(p)}\); these two identities give the stated first-infinity
cross-ratio and triple-product formulas after the same cancellation.
\end{proposition}

\begin{proof}
The identity \(H(\sigma Z,\sigma W)=H(Z,W)\) follows immediately from
\[
  H(Z,W)=Z_0\overline{W_2}+Z_1\overline{W_1}+Z_2\overline{W_0}.
\]
The exchange of \(\infty\) and the finite origin is the coordinate statement
\[
  \sigma(1,0,0)=(0,0,1),\qquad
  \sigma(0,0,1)=(1,0,0).
\]
For \(p=(x+i y,t)\), write \(\zeta=x+iy\).  Then
\[
  \widehat\xi(p)=\left(\frac{-r+it}{2},\zeta,1\right).
\]
The displayed formulas give
\[
  x'+iy'=\alpha(p)\zeta,\qquad
  |x'+iy'|^2=\frac{4r}{D},\qquad
  t'=-\frac{4t}{D}.
\]
Therefore the first two coordinates of \(\widehat\xi(I(p))\) are
\(\alpha(p)\) and \(\alpha(p)\zeta\).  The last coordinate agrees because
\[
  \alpha(p)\frac{-r+it}{2}
  =
  \frac{(-2r-2it)(-r+it)}{2D}
  =
  1.
\]
The denominator statement follows from
\[
  D=(x^2+y^2)^2+t^2,
\]
which vanishes exactly when \(x=y=t=0\).  Substituting the displayed formulas
for \(x',y',t'\) gives
\[
  x'^2+y'^2=\frac{4r}{D},
  \qquad
  D(I(p))=\left(\frac{4r}{D}\right)^2+
          \left(\frac{-4t}{D}\right)^2
        =\frac{16}{D}.
\]
A second substitution in the same formula for \(I\) then gives
\(I(I(p))=p\).  Together with the exchange of \(0\) and \(\infty\), this is
the compactified involution law.
\end{proof}

\section{Full CR/contact atlas}

The finite Heisenberg chart is not a privileged extra structure.  It is one
chart in the projective CR atlas of the null boundary.  The full atlas is
obtained by choosing, for each omitted boundary point, a projective-unitary
map carrying that point to the standard point at infinity.

\begin{proposition}[Unitary transitivity on null lines]
For any two boundary points \([\xi],[\xi']\in\partial\CH^2\), there is a
Hermitian-unitary map \(g\) such that
\[
  g[\xi]=[\xi'].
\]
\end{proposition}

\begin{proof}
Choose nonzero null representatives \(\xi,\xi'\).  Since the Hermitian form
is nondegenerate, choose \(u,u'\in V\) with
\[
  h(\xi,u)=1,\qquad h(\xi',u')=1.
\]
Set
\[
  \eta=u-\frac{h(u,u)}2\,\xi,
  \qquad
  \eta'=u'-\frac{h(u',u')}2\,\xi'.
\]
The diagonal values \(h(u,u)\) and \(h(u',u')\) are real.  Since \(\xi\) and
\(\xi'\) are null, the new vectors satisfy
\[
  h(\xi,\eta)=1,\qquad h(\eta,\eta)=0,
  \qquad
  h(\xi',\eta')=1,\qquad h(\eta',\eta')=0.
\]
The two planes \(\Span\{\xi,\eta\}\) and \(\Span\{\xi',\eta'\}\) have the
same Gram matrix
\[
  \begin{pmatrix}0&1\\1&0\end{pmatrix}.
\]
Their orthogonal complements are positive one-dimensional Hermitian spaces.
Choose unit positive vectors \(e\) and \(e'\) in these complements.  Then the
ordered bases
\[
  (\xi,e,\eta),\qquad(\xi',e',\eta')
\]
have the same Gram matrix
\[
  \begin{pmatrix}
    0&0&1\\
    0&1&0\\
    1&0&0
  \end{pmatrix}.
\]
The linear map sending the first ordered basis to the second therefore
preserves the Hermitian form and sends \([\xi]\) to \([\xi']\).
\end{proof}

\begin{definition}[Boundary charts with omitted pole]
In the Siegel model, let
\[
  \infty=[(1,0,0)].
\]
For each \(\omega\in\partial\CH^2\), choose a Hermitian-unitary map
\(g_\omega\) with
\[
  g_\omega\omega=\infty.
\]
Let
\[
  U_\omega=\partial\CH^2\setminus\{\omega\}.
\]
The chart
\[
  \psi_\omega:U_\omega\longrightarrow \C\times\mathbb R
\]
is defined by
\[
  g_\omega[\xi]=[\widehat\xi(\zeta,t)]
  \quad\Longleftrightarrow\quad
  \psi_\omega([\xi])=(\zeta,t).
\]
The family \(\{(U_\omega,\psi_\omega)\}_{\omega\in\partial\CH^2}\) is the
projective Heisenberg atlas of the boundary.
\end{definition}

\begin{proposition}[Atlas coverage and coordinate uniqueness]
The sets \(U_\omega\) cover \(\partial\CH^2\).  In each chart, the finite
coordinate \((\zeta,t)\) is unique.
\end{proposition}

\begin{proof}
Given a boundary point \(\xi\), choose any boundary point \(\omega\ne\xi\).
Then \(\xi\in U_\omega\), so the \(U_\omega\)'s cover the boundary.

For uniqueness in the standard chart, suppose
\[
  [\widehat\xi(\zeta,t)]=[\widehat\xi(\eta,s)].
\]
The third homogeneous coordinate of each finite lift is \(1\), so the
projective scalar must be \(1\).  Hence the second coordinates give
\(\zeta=\eta\), and the first coordinates then give \(t=s\).  Applying this
after \(g_\omega\) gives uniqueness in every chart.
\end{proof}

\begin{theorem}[Contact-conformal chart overlaps]
Let \(\omega,\omega'\in\partial\CH^2\), and let
\[
  \phi=\psi_{\omega'}\circ\psi_\omega^{-1}
\]
be a transition map on a nonempty finite overlap.  Let
\[
  G=g_{\omega'}g_\omega^{-1}.
\]
Then there is a nonzero scalar function \(\rho(p)\) on the overlap such that
\[
  \widehat\xi(\phi(p))=\rho(p)\,G\widehat\xi(p).
\]
For the standard finite contact form
\[
  \theta=\dd t+2\,\operatorname{Im}(\overline\zeta\,\dd\zeta),
\]
the overlap satisfies
\[
  \phi^\ast\theta=|\rho|^2\theta,
  \qquad
  \phi^\ast(\theta\wedge\dd\theta)=|\rho|^4\,\theta\wedge\dd\theta.
\]
Thus all projective-unitary boundary chart overlaps preserve the contact
distribution and transform the contact-volume class conformally.
\end{theorem}

\begin{proof}
The two finite lifts \(\widehat\xi(\phi(p))\) and \(G\widehat\xi(p)\)
represent the same projective null line, so they differ by a nonzero scalar
\(\rho(p)\).

Let \(s(p)=\widehat\xi(p)\).  In the finite chart,
\[
  \theta=2\,\operatorname{Im} H(\dd s,s).
\]
Differentiate the lift relation:
\[
  \dd(\rho Gs)=(\dd\rho)Gs+\rho\,G\dd s.
\]
Since \(Gs\) is null, the term involving \(\dd\rho\) pairs trivially with
\(\rho Gs\).  Because \(G\) is Hermitian-unitary,
\[
\begin{aligned}
  H(\dd(\rho Gs),\rho Gs)
  &=
  |\rho|^2 H(G\dd s,Gs)\\
  &=
  |\rho|^2 H(\dd s,s).
\end{aligned}
\]
Taking twice the imaginary part gives
\[
  \phi^\ast\theta=|\rho|^2\theta.
\]
Finally,
\[
  \dd(|\rho|^2\theta)
  =
  \dd(|\rho|^2)\wedge\theta+|\rho|^2\dd\theta,
\]
and the term
\(\theta\wedge\dd(|\rho|^2)\wedge\theta\) vanishes.  Therefore
\[
  \phi^\ast(\theta\wedge\dd\theta)
  =
  |\rho|^4\,\theta\wedge\dd\theta.
\]
\end{proof}

\begin{corollary}[Atlas covariance of boundary invariants]
With notation as in the preceding theorem, the boundary pair kernel satisfies
\[
  K(\phi(p),\phi(q))
  =
  \rho(p)\overline{\rho(q)}\,K(p,q).
\]
Consequently boundary cross-ratios are preserved by chart overlaps, and the
Cartan triple product is multiplied by the positive real norm factor
\[
  |\rho(p_1)|^2|\rho(p_2)|^2|\rho(p_3)|^2.
\]
In particular, the Cartan phase class is atlas independent.
\end{corollary}

\begin{proof}
The kernel identity is
\[
\begin{aligned}
K(\phi(p),\phi(q))
&=
H(\widehat\xi(\phi(p)),\widehat\xi(\phi(q)))\\
&=
H(\rho(p)G\widehat\xi(p),\rho(q)G\widehat\xi(q))\\
&=
\rho(p)\overline{\rho(q)}\,K(p,q).
\end{aligned}
\]
In a cross-ratio, each endpoint scalar occurs once in the numerator and once
in the denominator, so all four scalar factors cancel.  In a triple product,
the three factors multiply as
\[
  \rho_1\overline{\rho_2}\,
  \rho_2\overline{\rho_3}\,
  \rho_3\overline{\rho_1}
  =
  |\rho_1|^2|\rho_2|^2|\rho_3|^2,
\]
which is positive real.  Hence the phase class is unchanged.
\end{proof}

\begin{proposition}[Affine and inversion transitions]
The compactified affine Heisenberg similarities are chart transitions induced
by projective-unitary maps fixing \(\infty\).  The standard inversion is the
transition between the standard chart omitting
\(\infty\) and the chart omitting the finite origin, induced by the Siegel
coordinate swap
\[
  (Z_0,Z_1,Z_2)\longmapsto(Z_2,Z_1,Z_0).
\]
Thus the affine maps and the inversion computed above are concrete members of
the full projective CR/contact atlas.
\end{proposition}

\begin{proof}
The affine similarities computed above preserve the standard finite chart and
compactify by fixing \(\infty\).  Their contact weights and kernel weights are
the constant-weight cases of the general overlap theorem.

For the second statement, the coordinate swap sends the finite origin
\([(0,0,1)]\) to \(\infty=[(1,0,0)]\).  Hence it supplies the chart
\(\psi_0\) omitting the finite origin.  The preceding inversion proposition
proved that, on the common finite overlap, its coordinate expression is
exactly the displayed Heisenberg inversion.  Therefore inversion is a genuine
atlas transition, not an additional operation.
\end{proof}

\begin{definition}[Atlas CR circle]
An atlas CR circle is a boundary subset whose expression in one, equivalently
every, projective Heisenberg chart is the finite equation
\[
  H(\widehat\xi(\zeta,t),p)=0
\]
for some positive vector \(p\), together with the omitted chart point when
that point also satisfies the same projective incidence relation.
\end{definition}

\begin{theorem}[Chains are the CR circles of the atlas]
For \(p\in V_+\), the Hermitian chain
\[
  \mathcal C_p=\PH(p^\perp\cap V_0)
\]
is an atlas CR circle.  Conversely every atlas CR circle is a Hermitian chain.
In the chart \(\psi_\omega\), its finite equation is
\[
  H(\widehat\xi(\zeta,t),g_\omega p)=0.
\]
Writing \(g_\omega p=(p_0,p_1,p_2)\), this is
\[
  \frac{-|\zeta|^2+i t}{2}\overline{p_2}
  +\zeta\overline{p_1}
  +\overline{p_0}
  =
  0.
\]
The omitted point \(\omega\) lies on the chain exactly when \(h(\omega,p)=0\).
In the standard chart, this specializes to
\[
  \infty\in\mathcal C_p
  \quad\Longleftrightarrow\quad
  p_2=0.
\]
\end{theorem}

\begin{proof}
For \([\xi]\in U_\omega\), the chart definition gives
\[
  g_\omega[\xi]=[\widehat\xi(\zeta,t)].
\]
Since \(g_\omega\) is Hermitian-unitary,
\[
  h(\xi,p)=0
  \quad\Longleftrightarrow\quad
  H(g_\omega\xi,g_\omega p)=0.
\]
Replacing \(g_\omega\xi\) by the finite lift
\(\widehat\xi(\zeta,t)\) changes the left argument by a nonzero scalar and
does not change the zero locus.  Thus the finite chart equation is
\[
  H(\widehat\xi(\zeta,t),g_\omega p)=0,
\]
which expands to the displayed formula.  The omitted point belongs to the
chain precisely when it satisfies the same incidence relation
\(h(\omega,p)=0\).  In the standard chart the omitted point is
\((1,0,0)\), and
\[
  H((1,0,0),p)=\overline{p_2},
\]
so the condition is \(p_2=0\).

Conversely, any atlas CR circle is, by definition, given in some chart by
the equation
\[
  H(\widehat\xi(\zeta,t),q)=0
\]
with \(q\in V_+\).  Pulling back by the chosen chart map gives the Hermitian chain
\(\mathcal C_{g_\omega^{-1}q}\).  Hence the atlas CR circles and Hermitian
chains are the same objects.
\end{proof}

\section{Normalized ball kernel calculus}

In the ball model of \(\CH^2\), write \(z,w\in B^2\subset\C^2\) and use the
positive definite Euclidean Hermitian pairing \(\langle z,w\rangle\).  The
ball model is standard; the added object here is the left-normalized
phase/contrast factorization of its power kernels.  For a
nonzero exponent \(\gamma\), the power-law kernel is
\[
  K_\gamma(z,w)
  =
  (1-\langle z,w\rangle)^{-\gamma}.
\]
Its left-normalized overlap is
\[
  \Omega_\gamma(z,w)
  =
  \frac{K_\gamma(z,w)}{K_\gamma(z,z)}
  =
  \frac{(1-\langle z,w\rangle)^{-\gamma}}
       {(1-\langle z,z\rangle)^{-\gamma}}.
\]

Define
\[
  B(z,w)=1-\langle z,w\rangle,\qquad
  D(z)=B(z,z)=1-\langle z,z\rangle.
\]
On the open ball, \(D(z)>0\).  Cauchy--Schwarz gives \(B(z,w)\ne0\), so the
normalized overlap has the algebraic factorization
\[
  \Omega_\gamma(z,w)
  =
  \left(\frac{B(z,w)}{D(z)}\right)^{-\gamma}.
\]
The unit factor
\[
  U(z,w)=\frac{B(z,w)}{|B(z,w)|}
\]
records phase, while the positive factor
\[
  R_L(z,w)=\frac{|B(z,w)|}{D(z)}
\]
records the left-normalized radial/contrast part.  Thus
\[
  \Omega_\gamma(z,w)=U(z,w)^{-\gamma}R_L(z,w)^{-\gamma}.
\]
This is the kernel-side analogue of the pair/triple/contact split.

\begin{proposition}[Hilbert-ball kernel split]
Let \(B^2=\{z\in\C^2:\|z\|<1\}\), and set
\[
  B(z,w)=1-\langle z,w\rangle,\qquad
  D(z)=B(z,z).
\]
For \(z,w\in B^2\), one has
\[
  B(z,w)\ne0,\qquad D(z)\ne0.
\]
Consequently the functions
\[
  M(z,w)=|B(z,w)|,\qquad
  L(z)=D(z),\qquad
  R_L(z,w)=\frac{M(z,w)}{L(z)},\qquad
  U(z,w)=\frac{B(z,w)}{M(z,w)}
\]
are defined, and
\[
  \frac{B(z,w)^{-\gamma}}{D(z)^{-\gamma}}
  =
  U(z,w)^{-\gamma} R_L(z,w)^{-\gamma}.
\]
\end{proposition}

\begin{proof}
Since \(z,w\) lie in the open unit ball,
\(\|z\|<1\) and \(\|w\|<1\).  Cauchy--Schwarz gives
\[
  |\langle z,w\rangle|\le \|z\|\,\|w\|<1.
\]
Hence \(\langle z,w\rangle\ne1\), so \(B(z,w)\ne0\).  The diagonal case gives
\[
  D(z)=1-\|z\|^2>0.
\]
The displayed factorization is then the identity
\[
  \frac{B(z,w)}{D(z)}
  =
  \frac{B(z,w)}{|B(z,w)|}\,
  \frac{|B(z,w)|}{D(z)}
\]
raised to the power \(-\gamma\).
\end{proof}

\begin{proposition}[Normalized kernel overlap from the origin]
Let \(\gamma>0\), and let
\[
  K_\gamma(z,w)=(1-\langle z,w\rangle)^{-\gamma}
\]
on the unit ball \(B^2\subset\C^2\).  Define the symmetric normalized kernel
overlap with the origin by
\[
  \Phi_\gamma(z,0)
  =
  \frac{K_\gamma(z,0)}
       {\sqrt{K_\gamma(z,z)K_\gamma(0,0)}}.
\]
Then
\[
  \Phi_\gamma(z,0)=(1-\|z\|^2)^{\gamma/2},
\]
and the logarithmic overlap divergence is
\[
  E_\gamma(z)
  =
  -\log|\Phi_\gamma(z,0)|
  =
  -\frac{\gamma}{2}\log(1-\|z\|^2).
\]
For the ordinary Bergman kernel on \(B^2\), one has \(\gamma=3\), so
\[
  \Phi_3(z,0)=(1-\|z\|^2)^{3/2},
  \qquad
  E_3(z)=-\frac32\log(1-\|z\|^2).
\]
\end{proposition}

\begin{proof}
Since \(K_\gamma(z,0)=1\), \(K_\gamma(0,0)=1\), and
\[
  K_\gamma(z,z)=(1-\|z\|^2)^{-\gamma},
\]
the normalized overlap is
\[
  \Phi_\gamma(z,0)
  =
  \frac{1}{\sqrt{(1-\|z\|^2)^{-\gamma}}}
  =
  (1-\|z\|^2)^{\gamma/2}.
\]
Taking the negative logarithm of its absolute value gives the displayed
formula for \(E_\gamma\).  The Bergman kernel in complex dimension \(2\) is
\[
  K_B(z,w)=c_K(1-\langle z,w\rangle)^{-3},
\]
and the constant \(c_K\) cancels in the normalized overlap; hence the
Bergman case is \(\gamma=3\).
\end{proof}

\begin{proposition}[Radial divergence and asymptotic slope]
Let
\[
  \phi(z)=-\log(1-\|z\|^2),
  \qquad
  g^\ast=2\,\partial\overline{\partial}\phi
\]
with the real radial normalization
\[
  ds^2=\frac{4\,d\rho^2}{(1-\rho^2)^2},
  \qquad \rho=\|z\|.
\]
If \(r=d_{g^\ast}(0,z)\), then
\[
  r(\rho)=2\,\operatorname{arctanh}\rho
  =
  \log\frac{1+\rho}{1-\rho}.
\]
Consequently
\[
  1-\rho^2=\operatorname{sech}^2\!\left(\frac r2\right),
\]
and the logarithmic divergence along a radial geodesic is
\[
  E_\gamma(r)
  =
  \gamma\log\cosh\left(\frac r2\right).
\]
In particular,
\[
  E_\gamma(r)
  =
  \frac{\gamma}{2}r-\gamma\log2+o(1)
  =
  \frac{\gamma}{2}r+\mathcal O(1)
  \qquad (r\to\infty).
\]
For the Bergman kernel on \(B^2\), this gives
\[
  E_3(r)=3\log\cosh\left(\frac r2\right)
  =
  \frac32 r+\mathcal O(1).
\]
\end{proposition}

\begin{proof}
Write \(z=\rho\xi\), where \(\|\xi\|=1\) and \(0\le\rho<1\).  For
\[
  \phi(z)=-\log(1-\|z\|^2),
\]
one computes
\[
  \partial\overline{\partial}\phi
  =
  \sum_{i,j}
  \left(
    \frac{\delta_{ij}}{1-\|z\|^2}
    +
    \frac{\overline z_i z_j}{(1-\|z\|^2)^2}
  \right)
  dz_i\,d\overline z_j.
\]
Along the radial path \(z=\rho\xi\), \(dz_i=\xi_i\,d\rho\).  Hence
\[
\begin{aligned}
\partial\overline{\partial}\phi(\dot z,\dot z)
&=
\frac{1}{1-\rho^2}
+
\frac{\rho^2}{(1-\rho^2)^2}\\
&=
\frac{1}{(1-\rho^2)^2}.
\end{aligned}
\]
With the stated real radial normalization of \(g^\ast\), this is
\[
  ds^2=\frac{4\,d\rho^2}{(1-\rho^2)^2},
  \qquad
  ds=\frac{2\,d\rho}{1-\rho^2}.
\]
Therefore
\[
  r(\rho)
  =
  \int_0^\rho\frac{2\,du}{1-u^2}
  =
  2\,\operatorname{arctanh}\rho
  =
  \log\frac{1+\rho}{1-\rho}.
\]
Thus \(\rho=\tanh(r/2)\), and so
\[
  1-\rho^2
  =
  1-\tanh^2(r/2)
  =
  \operatorname{sech}^2(r/2).
\]
Substituting into
\[
  E_\gamma(\rho)
  =
  -\frac{\gamma}{2}\log(1-\rho^2)
\]
gives
\[
  E_\gamma(r)
  =
  -\frac{\gamma}{2}
  \log\operatorname{sech}^2(r/2)
  =
  \gamma\log\cosh(r/2).
\]
Finally,
\[
  \log\cosh(r/2)=\frac r2-\log2+o(1)
  \qquad (r\to\infty),
\]
which gives the asymptotic formula.
\end{proof}

\begin{theorem}[Kernel automorphy and normalized covariance]
Let \(g\) be a ball-model projective-unitary transformation with automorphy
factor \(J_g(z)\ne0\), so that
\[
  B(gz,gw)
  =
  \frac{B(z,w)}{J_g(z)\overline{J_g(w)}}.
\]
Set
\[
  A_L(z,w)=\frac{B(z,w)}{D(z)}.
\]
Then
\[
  D(gz)=\frac{D(z)}{|J_g(z)|^2},
  \qquad
  A_L(gz,gw)
  =
  \frac{\overline{J_g(z)}}{\overline{J_g(w)}}A_L(z,w).
\]
Consequently
\[
  K_\gamma(gz,gw)
  =
  \bigl(J_g(z)\overline{J_g(w)}\bigr)^\gamma K_\gamma(z,w),
\]
and
\[
  \Omega_\gamma(gz,gw)
  =
  \left(\frac{\overline{J_g(w)}}{\overline{J_g(z)}}\right)^\gamma
  \Omega_\gamma(z,w).
\]
The phase/contrast factors transform by
\[
  R_L(gz,gw)
  =
  \frac{|J_g(z)|}{|J_g(w)|}R_L(z,w),
\]
and
\[
  U(gz,gw)
  =
  \frac{|J_g(z)|\,|J_g(w)|}
       {J_g(z)\overline{J_g(w)}}\,U(z,w).
\]
\end{theorem}

\begin{proof}
Setting \(w=z\) in the automorphy law gives
\[
  D(gz)=B(gz,gz)
  =
  \frac{B(z,z)}{J_g(z)\overline{J_g(z)}}
  =
  \frac{D(z)}{|J_g(z)|^2}.
\]
Dividing the pair-base covariance by this diagonal covariance gives
\[
\begin{aligned}
A_L(gz,gw)
&=
\frac{B(gz,gw)}{D(gz)}\\
&=
\frac{B(z,w)}{J_g(z)\overline{J_g(w)}}\,
  \frac{|J_g(z)|^2}{D(z)}\\
&=
\frac{\overline{J_g(z)}}{\overline{J_g(w)}}A_L(z,w).
\end{aligned}
\]
Since \(K_\gamma=B^{-\gamma}\), the kernel automorphy formula is the same
identity raised to the power \(-\gamma\).  Dividing by the diagonal kernel
formula gives the displayed law for \(\Omega_\gamma\).

For the contrast factor,
\[
  |B(gz,gw)|
  =
  \frac{|B(z,w)|}{|J_g(z)|\,|J_g(w)|},
\]
and the formula for \(D(gz)\) gives
\[
  R_L(gz,gw)
  =
  \frac{|B(gz,gw)|}{D(gz)}
  =
  \frac{|J_g(z)|}{|J_g(w)|}R_L(z,w).
\]
Finally,
\[
  U(gz,gw)
  =
  \frac{B(gz,gw)}{|B(gz,gw)|}
  =
  \frac{|J_g(z)|\,|J_g(w)|}
       {J_g(z)\overline{J_g(w)}}\,U(z,w).
\]
\end{proof}

\begin{proposition}[Boundary limit of the ball pair base]
Use the ball lift
\[
  \widetilde z=(z_1,z_2,1),
\]
for the Hermitian form
\[
  h_B(Z,W)=Z_1\overline{W_1}+Z_2\overline{W_2}-Z_3\overline{W_3}.
\]
Then
\[
  B(z,w)=-h_B(\widetilde z,\widetilde w).
\]
Let \(\zeta,\eta\in\partial B^2\), set
\[
  \widetilde\zeta=(\zeta_1,\zeta_2,1),
  \qquad
  \widetilde\eta=(\eta_1,\eta_2,1),
\]
and let \(z_r=r\zeta\), \(w_s=s\eta\) with \(0<r,s<1\).  Then
\[
  \lim_{r,s\to1}B(z_r,w_s)
  =
  1-\langle\zeta,\eta\rangle
  =
  -h_B(\widetilde\zeta,\widetilde\eta).
\]
If \(C\) is a projective-unitary Cayley map from the ball form to the Siegel
form, and
\[
  C\widetilde\zeta=\lambda_\zeta\widehat\xi(p),
  \qquad
  C\widetilde\eta=\lambda_\eta\widehat\xi(q),
\]
with \(\lambda_\zeta,\lambda_\eta\ne0\), then
\[
  \lim_{r,s\to1}B(z_r,w_s)
  =
  -\lambda_\zeta\overline{\lambda_\eta}\,
  H(\widehat\xi(p),\widehat\xi(q)).
\]
Equivalently, the projectively normalized boundary limit is the Heisenberg
pair kernel:
\[
  -\lambda_\zeta^{-1}\overline{\lambda_\eta}^{-1}
  \lim_{r,s\to1}B(z_r,w_s)
  =
  H(\widehat\xi(p),\widehat\xi(q)).
\]
\end{proposition}

\begin{proof}
The ball-lift identity is immediate:
\[
  h_B(\widetilde z,\widetilde w)
  =
  \langle z,w\rangle-1
  =
  -B(z,w).
\]
Taking \(z_r=r\zeta\) and \(w_s=s\eta\) gives
\[
  B(z_r,w_s)=1-rs\,\langle\zeta,\eta\rangle,
\]
and the stated limit follows as \(r,s\to1\).

Since \(C\) preserves the Hermitian form,
\[
  h_B(\widetilde\zeta,\widetilde\eta)
  =
  H(C\widetilde\zeta,C\widetilde\eta).
\]
Substituting the projective boundary representatives gives
\[
  H(C\widetilde\zeta,C\widetilde\eta)
  =
  \lambda_\zeta\overline{\lambda_\eta}\,
  H(\widehat\xi(p),\widehat\xi(q)).
\]
Combining this with the preceding sign relation proves both displayed
boundary formulas.
\end{proof}

\begin{corollary}[Boundary readout compatibility]
On boundary configurations where the relevant pair kernels are nonzero, the
renormalized ball-boundary pair bases give the same cross-ratios as the
Heisenberg pair kernel.  Triple products differ only by the product of the
projective norm factors and the fixed sign convention, so their Cartan phase
class is the same after the projective normalization in the preceding
proposition.
\end{corollary}

\begin{proof}
For each boundary pair, the preceding proposition writes the ball-boundary
base as a nonzero scalar attached to the first endpoint times the conjugate
scalar attached to the second endpoint times the Heisenberg pair kernel,
with the same fixed sign convention for every pair.  In a cross-ratio each
endpoint scalar and each fixed sign occurs once in the numerator and once in
the denominator, so they cancel.  In a triple product the endpoint scalars
multiply to a product of norm factors.  The normalized boundary convention
removes the fixed sign, leaving the same Cartan phase class.
\end{proof}

\section{Differential kernel algebra}

The kernel calculus also gives local second-order objects: Hessian blocks,
radial pole orders, and finite-dimensional reductions obtained by eliminating
one variable.  The following elementary identities are stated separately
because they are useful whenever the ball kernel is expanded near a chosen
base point or near the boundary.

\begin{lemma}[Rank-one Schur decoupling]
Let \(N\) be a finite-dimensional Hermitian vector space, let \(R\cong\C\),
and let \(B:N\to R\) be a nonzero complex linear map.  For a self-adjoint
operator \(A:N\to N\), a scalar \(m\in\mathbb R\), and a parameter
\(\lambda\ne m\), define the Schur operator
\[
  S(\lambda)
  =
  A-\frac{1}{m-\lambda}B^\ast B.
\]
Then
\[
  B^\ast B|_{\ker B}=0,
  \qquad
  \operatorname{rank}(B^\ast B)=1,
\]
and the Schur correction acts only on the one-dimensional subspace
\[
  (\ker B)^\perp=\operatorname{im}B^\ast.
\]
Equivalently, the directions in \(\ker B\) are unchanged by the rank-one
correction:
\[
  S(\lambda)v=Av
  \qquad (v\in\ker B).
\]
\end{lemma}

\begin{proof}
For \(v\in\ker B\), one has \(Bv=0\), and therefore
\[
  B^\ast Bv=B^\ast 0=0.
\]
Since \(B:N\to R\cong\C\) is nonzero, its image has dimension \(1\).  Hence
\(B^\ast B\) has image contained in \(\operatorname{im}B^\ast\), which is at
most one-dimensional.  It is not zero: choosing \(u\) with \(Bu\ne0\), one has
\[
  \langle B^\ast Bu,u\rangle_N
  =
  \langle Bu,Bu\rangle_R
  =
  |Bu|^2>0.
\]
Thus \(\operatorname{rank}(B^\ast B)=1\).

Finally, for every \(v\in\ker B\) and every \(n\in N\),
\[
  \langle v,B^\ast n\rangle_N
  =
  \langle Bv,n\rangle_R
  =
  0,
\]
so \(\operatorname{im}B^\ast\subset(\ker B)^\perp\).  Both spaces are
one-dimensional because \(B\ne0\), hence they are equal.  The displayed formula
for \(S(\lambda)v\) follows immediately from \(B^\ast Bv=0\).
\end{proof}

\begin{proposition}[Signed two-plane Hessian block]
Let \(\omega,\mu\ge0\) and \(c\in\C\).  The Hermitian matrix
\[
  M=
  \begin{pmatrix}
    -\omega^2 & c\\
    \overline c & \mu^2
  \end{pmatrix}
\]
has eigenvalues
\[
  \lambda_\pm
  =
  \frac{\mu^2-\omega^2}{2}
  \pm
  \frac12
  \sqrt{(\omega^2+\mu^2)^2+4|c|^2}.
\]
Thus
\[
  \lambda_+\ge0,\qquad \lambda_-\le0.
\]
Moreover, \(\lambda_+>0\) exactly when \(\mu>0\) or \(c\ne0\), and
\(\lambda_-<0\) exactly when \(\omega>0\) or \(c\ne0\).  If \(c=|c|e^{i\varphi}\),
then the diagonal unitary change of basis
\[
  P=\begin{pmatrix}e^{i\varphi}&0\\0&1\end{pmatrix}
\]
converts \(M\) to the real symmetric block
\[
  P^\ast M P=
  \begin{pmatrix}
    -\omega^2 & |c|\\
    |c| & \mu^2
  \end{pmatrix}.
\]
When \(\omega^2+\mu^2>0\), its real mixing angle may be chosen to satisfy
\[
  \tan(2\theta)
  =
  -\,\frac{2|c|}{\omega^2+\mu^2}.
\]
\end{proposition}

\begin{proof}
The characteristic polynomial is
\[
\begin{aligned}
0
&=
\det
\begin{pmatrix}
  -\omega^2-\lambda & c\\
  \overline c & \mu^2-\lambda
\end{pmatrix}  \\
&=
(-\omega^2-\lambda)(\mu^2-\lambda)-|c|^2  \\
&=
\lambda^2+(\omega^2-\mu^2)\lambda-\omega^2\mu^2-|c|^2.
\end{aligned}
\]
The quadratic formula gives
\[
  \lambda
  =
  \frac{\mu^2-\omega^2}{2}
  \pm
  \frac12
  \sqrt{(\omega^2-\mu^2)^2+4\omega^2\mu^2+4|c|^2},
\]
which is the stated expression because
\[
  (\omega^2-\mu^2)^2+4\omega^2\mu^2=(\omega^2+\mu^2)^2.
\]

Since
\[
  \sqrt{(\omega^2+\mu^2)^2+4|c|^2}
  \ge |\mu^2-\omega^2|,
\]
one has \(\lambda_+\ge0\) and \(\lambda_-\le0\).  The equality cases follow
directly from the same formula: \(\lambda_+=0\) exactly when
\(\mu=0\) and \(c=0\), while \(\lambda_-=0\) exactly when
\(\omega=0\) and \(c=0\).

The unitary conjugation satisfies
\[
  e^{-i\varphi}c=|c|,
  \qquad
  \overline c\,e^{i\varphi}=|c|,
\]
and therefore gives the displayed real symmetric block.  For a real symmetric
matrix
\[
  \begin{pmatrix}a&b\\b&d\end{pmatrix},
\]
the diagonalizing angle satisfies
\[
  \tan(2\theta)=\frac{2b}{a-d}.
\]
Here \(a=-\omega^2\), \(b=|c|\), and \(d=\mu^2\), giving
\[
  \tan(2\theta)
  =
  \frac{2|c|}{-\omega^2-\mu^2}
  =
  -\frac{2|c|}{\omega^2+\mu^2}.
\]
\end{proof}

\begin{proposition}[Radial pole-order scaling tower]
Let
\[
  s=\|z\|^2,\qquad
  \delta=1-s,
\]
and for \(\gamma>0\) define the logarithmic radial contrast
\[
  \Theta_\gamma(s)=-\frac{\gamma}{2}\log(1-s).
\]
Then every boundary pole order is an exponential in \(\Theta_\gamma\):
\[
  \delta^{-n}
  =
  \exp\!\left(\frac{2n}{\gamma}\Theta_\gamma\right)
  \qquad (n\ge0).
\]
Moreover the radial derivatives of the kernel potential
\[
  \mathcal L_\gamma(s)=-\gamma\log(1-s)
\]
are
\[
  \frac{\dd^n \mathcal L_\gamma}{\dd s^n}
  =
  \gamma (n-1)!(1-s)^{-n}
  \qquad (n\ge1).
\]
For the Bergman kernel on \(B^2\), \(\gamma=3\), and hence
\[
  \delta^{-n}=\exp\!\left(\frac{2n}{3}\Theta_3\right),
  \qquad
  \frac{\dd^n \mathcal L_3}{\dd s^n}
  =
  3 (n-1)!\delta^{-n}.
\]
\end{proposition}

\begin{proof}
From the definition of \(\Theta_\gamma\),
\[
  \Theta_\gamma=-\frac{\gamma}{2}\log\delta.
\]
Therefore
\[
  \log\delta=-\frac{2}{\gamma}\Theta_\gamma,
\]
and so
\[
  \delta^{-n}
  =
  \exp(-n\log\delta)
  =
  \exp\!\left(\frac{2n}{\gamma}\Theta_\gamma\right).
\]

For the derivative formula, first compute
\[
  \frac{\dd \mathcal L_\gamma}{\dd s}
  =
  \frac{\gamma}{1-s}.
\]
Repeated differentiation gives, by induction,
\[
  \frac{\dd^n}{\dd s^n}\bigl[-\gamma\log(1-s)\bigr]
  =
  \gamma (n-1)!(1-s)^{-n}
  \qquad (n\ge1).
\]
Substituting \(\gamma=3\) gives the Bergman specialization.
\end{proof}

\section{Synopsis of the calculus}

The preceding sections give a closed algebraic calculus for the basic
incidence, metric, angular, boundary, and kernel structures of \(\CH^2\).
It is useful to collect the principal objects before passing to the finite
Hermitian analogue.

\begin{theorem}[Intrinsic invariant package]
Let \(V\) be a Hermitian vector space of signature \((2,1)\).  The following
data are projectively defined on their stated domains.
\begin{description}
\item[Projective strata.]
\[
  \CH^2=\PH(V_-),\qquad \partial\CH^2=\PH(V_0),
\]
and positive lines are polar representatives of complex geodesics and their
boundary chains.

\item[Pair calculus.]
For non-null representatives \(z,w\),
\[
  \eta_H([z],[w])
  =
  \frac{h(z,w)h(w,z)}{h(z,z)h(w,w)},
  \qquad
  q_H([z],[w])=1-\eta_H([z],[w]),
\]
and the denominator-cleared form is
\[
  q_H([z],[w])h(z,z)h(w,w)=\Delta_H(z,w).
\]

\item[Polar spread.]
For positive poles \(p,q\), the same determinant quotient gives a
projective spread of the corresponding complex geodesics.  The vanishing of
\(\Delta_H(p,q)\) detects boundary tangency of the associated chains.

\item[Triple phase.]
For three non-null points, the triple product
\[
  T_H(z_1,z_2,z_3)
  =
  -h(z_1,z_2)h(z_2,z_3)h(z_3,z_1)
\]
has a projective phase class.  Cyclic permutation preserves this class and
reversal conjugates it.

\item[Triangle trigonometry.]
The three-point Gram determinant yields the Hermitian triangle constraint,
the cyclic vertex-spread law, and the Pythagoras specialization.  Thus the
side quadrances and the triple phase jointly determine the algebraic
triangle relations.

\item[Cross-ratio calculus.]
For quadruples with nonzero denominator,
\[
  X(a,b;c,d)=\frac{h(c,a)h(d,b)}{h(d,a)h(c,b)}
\]
is projectively defined and satisfies endpoint inversion, double-swap
preservation, and ordered-pair conjugation.

\item[Boundary contact calculus.]
In Heisenberg coordinates, the finite boundary chart has pair kernel
\[
  H(p,q)=\frac{|\zeta-\eta|^2-i(t-s+2\,\operatorname{Im}(\zeta\overline\eta))}{2},
\]
contact form
\[
  \theta=\dd t+2(x\,\dd y-y\,\dd x),
\]
and contact density
\[
  \theta\wedge\dd\theta=4\,\dd t\wedge\dd x\wedge\dd y.
\]
Heisenberg translations and unit rotations preserve \(\theta\), while the
dilation \((\zeta,t)\mapsto(r\zeta,r^2t)\) scales \(\theta\) by \(r^2\).

\item[Kernel calculus.]
In the ball model,
\[
  K_\gamma(z,w)=(1-\langle z,w\rangle)^{-\gamma}
\]
has a left-normalized phase/radial split, automorphy covariance under
projective-unitary maps, phase/contrast covariance, and a boundary limit
agreeing projectively with the Heisenberg pair kernel.
\end{description}
\end{theorem}

\begin{proof}
Each assertion is the projective descent of a formula proved above.  Pair,
polar-spread, triple, and cross-ratio quantities have balanced Hermitian
scale weights, so they descend to projective classes on the stated domains.
The triangle relations are the three-point Gram identity and its cofactor
specializations.  The contact and kernel formulas are the coordinate
expressions derived in the Heisenberg and ball models.
\end{proof}

\section{Finite Hermitian incidence}

The same algebraic formulas make sense over a field with involution.  For a
quadratic finite-field extension \(k\subset \ell\), the involution is
Frobenius conjugation and the complex phase quotient is replaced by the norm
quotient for
\[
  N(a)=a\overline a.
\]
Inequalities such as positive and negative are replaced by polarity strata:
null lines and anisotropic lines.

\begin{theorem}[Finite Hermitian unital]
Let \(k=\mathbb F_q\), let \(\ell=\mathbb F_{q^2}\), and let
\(\overline a=a^q\).  Let \(V=\ell^3\) carry a nondegenerate Hermitian form
\(h\) of Witt index one.  Define
\[
  \partial_q=\{[x]\in\PH(V):h(x,x)=0\},
  \qquad
  \mathcal P_q=\{[p]\in\PH(V):h(p,p)\ne0\}.
\]
For \([p]\in\mathcal P_q\), define the chain
\[
  \mathcal C_p=\{[x]\in\partial_q:h(x,p)=0\}.
\]
Then
\[
  |\partial_q|=q^3+1,
  \qquad
  |\mathcal P_q|=q^2(q^2-q+1).
\]
Each chain contains \(q+1\) boundary points, each boundary point lies on
\(q^2\) chains, and any two distinct boundary points lie on a unique chain.
Equivalently, the null points and anisotropic Hermitian poles form the
finite Hermitian unital of order \(q\), a
\[
  2-(q^3+1,q+1,1)
\]
incidence structure.
\end{theorem}

\begin{proof}
All nondegenerate rank-three Hermitian forms of Witt index one over
\(\ell/k\) are equivalent, so use the standard split form
\[
  h(z,w)=z_0\overline{w_2}+z_1\overline{w_1}+z_2\overline{w_0}.
\]
If \(z_2=0\), the null equation gives \(z_1\overline z_1=0\), hence the
single point \([1,0,0]\).  If \(z_2\ne0\), scale to \(z=(x,y,1)\).  The null
equation is
\[
  x+\overline x+y\overline y=0.
\]
The trace map \(x\mapsto x+\overline x\) from \(\ell\) to \(k\) is onto and
has kernel of size \(q\).  For each of the \(q^2\) choices of \(y\), the
right-hand side \(-y\overline y\in k\) therefore has exactly \(q\) preimages
\(x\).  Thus there are \(q^3\) finite null points and one point at infinity,
so \(|\partial_q|=q^3+1\).

The projective plane \(\PH(V)\cong\mathbf{PG}(2,q^2)\) has
\[
  \frac{q^6-1}{q^2-1}=q^4+q^2+1
\]
points.  Subtracting the null points gives
\[
  |\mathcal P_q|=q^4+q^2+1-(q^3+1)=q^2(q^2-q+1).
\]

For a nonisotropic pole \(p\), the plane \(p^\perp\) is a nondegenerate
Hermitian plane.  Its isotropic projective lines are counted by a
two-dimensional norm equation.  In an orthogonal basis \(e_1,e_2\) of
\(p^\perp\),
\[
  h(u e_1+v e_2,u e_1+v e_2)=\alpha N(u)+\beta N(v),
  \qquad \alpha,\beta\in k^\times.
\]
Every isotropic projective line has \(v\ne0\), and after scaling to
\(v=1\) the equation becomes
\[
  N(u)=-\beta/\alpha.
\]
The finite-field norm \(\ell^\times\to k^\times\) is surjective and each
nonzero fiber has cardinality
\[
  \frac{q^2-1}{q-1}=q+1.
\]
Hence \(|\mathcal C_p|=q+1\).

Two distinct null points span a projective line containing two null points,
so it is not tangent to the Hermitian curve.  By Hermitian polarity it has a
unique anisotropic pole, and this pole is the unique chain incident with both
points.  Finally, the number of chains through a fixed point is obtained from
the punctured-row count: the \(q\) remaining points on each incident chain
are in bijection with the \(q^3\) boundary points different from the fixed
one.  Hence the number of incident chains is \(q^3/q=q^2\).
\end{proof}

\begin{corollary}[Finite incidence double count]
The number of incident boundary-point/chain pairs satisfies
\[
  q^2(q^2-q+1)(q+1)=(q^3+1)q^2.
\]
\end{corollary}

\begin{proof}
Count incidences first by chains and then by boundary points.  The algebraic
identity follows from
\[
  (q^2-q+1)(q+1)=q^3+1.
\]
\end{proof}

\begin{proposition}[Order-two diagonal model]
Let \(K=\mathbb F_4\), with involution \(a^\ast=a^2\), and let
\[
  h(z,w)=z_0w_0^\ast+z_1w_1^\ast+z_2w_2^\ast
\]
on \(K^3\).  Then \(N(a)=aa^\ast\) is \(0\) at \(a=0\) and \(1\) at every
nonzero scalar.  Consequently a nonzero vector is isotropic exactly when it
has two nonzero coordinates.  It is anisotropic exactly when it has one or
three nonzero coordinates.
\end{proposition}

\begin{proof}
The group \(K^\times\) has order \(3\), so \(a^{1+2}=a^3=1\) for every
nonzero \(a\).  Therefore
\[
  h(z,z)=N(z_0)+N(z_1)+N(z_2)
\]
is the parity, in characteristic two, of the number of nonzero coordinates.
The zero vector is excluded projectively.  Thus even nonzero support means
two nonzero coordinates, and odd support means one or three.
\end{proof}

\begin{corollary}[Order-two counts]
In the preceding \(\mathbb F_4\) model, the null-boundary quotient has
\(9\) points and the anisotropic-pole quotient has \(12\) points.
\end{corollary}

\begin{proof}
There are three two-coordinate support patterns and each nonzero coordinate
has three choices, so there are \(3\cdot3^2=27\) nonzero isotropic
representatives.  There are \(3\cdot3+3^3=36\) anisotropic representatives.
Projective fibers are free \(\mathbb F_4^\times\)-orbits of size \(3\).
Thus the quotient cardinalities are \(27/3=9\) and \(36/3=12\).
\end{proof}

\begin{example}[A finite \(\mathbb F_4\) triangle calculation]
Let
\[
  \mathbb F_4=\{0,1,\omega,\omega^2\},
  \qquad
  \omega^2+\omega+1=0,
\]
with involution \(a^\ast=a^2\).  Thus
\[
  \omega^\ast=\omega^2,\qquad
  (\omega^2)^\ast=\omega,\qquad
  \omega\omega^2=1.
\]
Use the diagonal Hermitian form
\[
  h(z,w)=z_0w_0^\ast+z_1w_1^\ast+z_2w_2^\ast.
\]
Consider the three vectors
\[
  u=(0,0,1),\qquad
  v=(1,1,1),\qquad
  w=(1,1,\omega).
\]
They represent three distinct non-null projective points.  Indeed,
\[
  h(u,u)=1,\qquad h(v,v)=1+1+1=1,
  \qquad h(w,w)=1+1+\omega\omega^\ast=1.
\]
The pairings around the oriented triple are
\[
  h(u,v)=1,\qquad
  h(v,w)=1+1+\omega^2=\omega^2,
  \qquad
  h(w,u)=\omega.
\]
The reverse pairings are their conjugates:
\[
  h(v,u)=1,\qquad h(w,v)=\omega,\qquad h(u,w)=\omega^2.
\]
Therefore the determinant quadrances are
\[
\begin{aligned}
  \Delta_H(u,v)&=1\cdot1-h(u,v)h(v,u)=1-1=0,\\
  \Delta_H(v,w)&=1\cdot1-h(v,w)h(w,v)=1-\omega^2\omega=0,\\
  \Delta_H(w,u)&=1\cdot1-h(w,u)h(u,w)=1-\omega\omega^2=0.
\end{aligned}
\]
Since all three diagonal products are \(1\), this gives
\[
  q_H([u],[v])=q_H([v],[w])=q_H([w],[u])=0.
\]
The finite triple product is nevertheless nonzero:
\[
\begin{aligned}
  T_H(u,v,w)
  &=-h(u,v)h(v,w)h(w,u)\\
  &=-(1)(\omega^2)(\omega)=-1=1,
\end{aligned}
\]
because the characteristic is \(2\).  Thus this small finite triangle already
shows the Hermitian feature that pair determinants do not exhaust the
three-point data: the triple-product class remains a separate invariant.
\end{example}

\section{Polynomial identity catalogue}

The preceding sections prove the calculus in context.  This catalogue records
the same information in denominator-cleared form, before analytic distance,
argument, or limiting readouts are applied.  It is the Wildberger-style
lookup layer of the CH2 calculus: each entry is a polynomial identity or an
incidence equation in the Hermitian form.

\begin{theorem}[Denominator-cleared identity catalogue]
The following identities are the core polynomial catalogue of the calculus.

\begin{description}
\item[Two-point Gram.]
For non-null \(z,w\),
\[
  \Delta_H(z,w)=h(z,z)h(w,w)-h(z,w)h(w,z),
  \qquad
  q_H([z],[w])h(z,z)h(w,w)=\Delta_H(z,w).
\]

\item[Three-point Gram.]
For \(a_i=h(z_i,z_i)\), \(h_{ij}=h(z_i,z_j)\), and
\(\tau=h_{12}h_{23}h_{31}\),
\[
\begin{aligned}
D_{123}
&=a_1a_2a_3
-a_3h_{12}h_{21}
-a_1h_{23}h_{32}
-a_2h_{31}h_{13}\\
&\qquad+\tau+\overline{\tau}.
\end{aligned}
\]

\item[Triangle constraint.]
On the non-null triple domain,
\[
  D_{123}
  =
  a_1a_2a_3
  \bigl(q_{12}+q_{23}+q_{31}-2+\Theta_{123}\bigr),
\]
and
\[
  \Omega_{123}\overline{\Omega}_{123}
  =
  (1-q_{12})(1-q_{23})(1-q_{31}).
\]

\item[Spread numerators.]
For cyclic \(\{i,j,k\}=\{1,2,3\}\),
\[
  \Sigma_i
  =
  \Delta_{ji}\Delta_{ik}
  -|h_{ji}h_{ik}-a_ih_{jk}|^2
  =
  a_iD_{123}.
\]
Equivalently,
\[
  \Delta_{ji}\Delta_{ik}s_i=a_iD_{123}.
\]

\item[Pythagoras.]
If the two side polars meeting at \(z_i\) are orthogonal, equivalently
\(s_i=1\), then
\[
  \Delta_{ji}\Delta_{ik}=a_iD_{123}.
\]

\item[Ideal chain test.]
For pairwise transverse boundary representatives \(\xi_1,\xi_2,\xi_3\),
\[
  D_\partial
  =
  h_{12}h_{23}h_{31}
  +
  \overline{h_{12}h_{23}h_{31}}.
\]
The three boundary points lie on one chain if and only if
\[
  D_\partial=0.
\]

\item[Cross-ratio symmetries.]
For
\[
  N(a,b;c,d)=h(c,a)h(d,b),
  \qquad
  D(a,b;c,d)=h(d,a)h(c,b),
\]
the endpoint laws before division are
\[
  N(a,b;d,c)=D(a,b;c,d),
  \qquad
  D(a,b;d,c)=N(a,b;c,d),
\]
\[
  N(b,a;c,d)=D(a,b;c,d),
  \qquad
  D(b,a;c,d)=N(a,b;c,d),
\]
\[
  N(b,a;d,c)=N(a,b;c,d),
  \qquad
  D(b,a;d,c)=D(a,b;c,d),
\]
\[
  N(c,d;a,b)=\overline{N(a,b;c,d)},
  \qquad
  D(c,d;a,b)=\overline{D(a,b;c,d)}.
\]
Where denominators are nonzero, these give the quotient laws for
\(X=N/D\):
\[
  X(a,b;c,d)X(a,b;d,c)=1,
  \qquad
  X(a,b;c,d)X(b,a;c,d)=1,
\]
\[
  X(b,a;d,c)=X(a,b;c,d),
  \qquad
  X(c,d;a,b)=\overline{X(a,b;c,d)}.
\]

\item[Chain incidence and tangency.]
For a positive pole \(p\),
\[
  [x]\in L_p\ \Longleftrightarrow\ h(x,p)=0,
  \qquad
  [\xi]\in\mathcal C_p\ \Longleftrightarrow\ h(\xi,p)=0.
\]
For two positive poles \(p,q\),
\[
  \mathcal C_p\text{ is tangent to }\mathcal C_q
  \quad\Longleftrightarrow\quad
  \Delta_H(p,q)=0.
\]

\item[Unique chain through a boundary pair.]
For distinct boundary points \([\xi]\ne[\eta]\),
\[
  p=\Span\{\xi,\eta\}^{\perp}
\]
is a positive line, and it is the unique pole satisfying
\[
  h(\xi,p)=h(\eta,p)=0.
\]

\item[Bisectors and spinal spheres.]
For non-null foci \(a,b\),
\[
  \mathcal B_{a,b}(x)
  =
  h(x,a)h(a,x)h(b,b)-h(x,b)h(b,x)h(a,a).
\]
The bisector is \(\mathcal B_{a,b}(x)=0\).  The spinal sphere is the same
equation restricted to \(h(x,x)=0\).  In the finite Heisenberg chart,
\[
  n_b A_a(\zeta,t)\overline{A_a(\zeta,t)}
  -
  n_a A_b(\zeta,t)\overline{A_b(\zeta,t)}
  =
  0.
\]

\item[Kernel covariance.]
For ball base \(B(z,w)=1-\langle z,w\rangle\),
\[
  B(gz,gw)=\frac{B(z,w)}{J_g(z)\overline{J_g(w)}},
  \qquad
  A_L(gz,gw)
  =
  \frac{\overline{J_g(z)}}{\overline{J_g(w)}}A_L(z,w).
\]

\item[Finite norm fibers and counts.]
Over \(\mathbb F_{q^2}/\mathbb F_q\), the boundary chart equation is
\[
  x+\overline x+y\overline y=0.
\]
For a nonisotropic pole row in a diagonalized Hermitian plane,
\[
  \alpha N(u)+\beta N(v)=0,
  \qquad
  N(r)=-\beta/\alpha
  \quad(v=1).
\]
The incidence double count is
\[
  q^2(q^2-q+1)(q+1)=(q^3+1)q^2.
\]
\end{description}
\end{theorem}

\begin{proof}
Each entry is a restatement of an identity already proved above, with
analytic divisions cleared when possible.  The two- and three-point Gram
entries are the Gram determinant definitions and the Three-point Hermitian
Gram Identity.  The triangle, spread, and Pythagoras entries are the
Hermitian triangle constraint, cyclic spread law, and Pythagoras corollary.
The ideal test is the boundary specialization of the same Gram determinant.
The cross-ratio laws are the endpoint symmetry propositions with denominators
assumed nonzero.  The chain, unique-chain, bisector, spinal-sphere, kernel,
and finite entries are exactly the incidence, covariance, bisector, boundary
trace, automorphy, and finite-Hermitian formulas proved in their respective
sections.
\end{proof}

\section{Conclusion}

Complex hyperbolic two-space is governed by Hermitian rather than symmetric
bilinear algebra.  Its algebraic trigonometry therefore requires more than a
single quadrance/spread pair: the natural projective data are pair modulus,
determinant quadrance, polar spread, triple phase, cross-ratio, CR/contact
boundary structure, normalized ball kernels, and finite Hermitian incidence.

The identities above show that these objects form a common algebraic system.
The same Hermitian form supplies the two-point Gram determinant, the
three-point triangle constraint, the chain and bisector equations, the
boundary contact kernels, the normalized ball-kernel covariance laws, and the
finite unital counts.  In this sense the calculus is intrinsic to
\(\CH^2\): it is the projective Hermitian geometry of negative, null, and
positive lines, written in denominator-cleared form whenever possible.

The motivation for isolating this particular package of invariants also comes
from the Distinguished Reproducing Kernel and Fixed Point Law programs, where
kernel normalization, boundary contact data, phase readouts, and finite
incidence all appear as organizing themes \cite{drk,fpl}.  The role of those
works here is motivational: the definitions and identities above are
developed within complex hyperbolic geometry itself.

\end{document}